\documentclass[12pt]{amsart}

\usepackage{amssymb,mathrsfs,amsmath,amsthm,color,bm,mathtools,bbm,wasysym}
\usepackage[noadjust]{cite}

\usepackage{enumerate}
\usepackage[centering]{geometry}
\theoremstyle{plain}
\newtheorem{thm}{Theorem}[section]
\newtheorem{defn}[thm]{Definition}

\newtheorem{prop}[thm]{Proposition}
\newtheorem{cor}[thm]{Corollary}
\newtheorem{lem}[thm]{Lemma}

\theoremstyle{remark}
\newtheorem{rem}{Remark}
\numberwithin{equation}{section}

\DeclareMathOperator{\hdim}{\dim_H}

\newcommand{\rc}{\mathcal R}
\newcommand{\rca}{\mathcal R_{\alpha}}
\newcommand{\rcat}{\mathcal R_{\alpha}^\times}
\newcommand{\rci}{\mathcal R_{\alpha,i}}
\newcommand{\dist}{\mathrm{dist}}

\newcommand{\N}{\mathbb N}
\newcommand{\R}{\mathbb R}
\newcommand{\Z}{\mathbb Z}

\renewcommand{\hm}{\mathcal H}
\newcommand{\nc}{\mathcal M_\infty}
\newcommand{\hc}{\mathcal H_\infty}
\newcommand{\cu}{\mathcal U}
\newcommand{\cq}{\mathcal Q}
\newcommand{\cw}{\mathcal W}
\newcommand{\ck}{\mathcal K}

\newcommand{\ca}{\mathcal A}
\newcommand{\ch}{\mathcal H}
\newcommand{\cf}{\mathcal F}
\newcommand{\cg}{\mathcal G}
\newcommand{\scg}{\mathscr G}

\newcommand{\qaq}{\mathrm{\quad and\quad}}

\begin{document}
		\title[A unified approach to MTP and LIP]{A unified approach to mass transference principle and large intersection property}
	\author{Yubin He}

	\address{Department of Mathematics, Shantou University, Shantou, Guangdong, 515063, China}
%	\address{Department of Mathematics, South China University of Technology,	Guangzhou, Guangdong 510641, P.~R.\ China}

	\email{ybhe@stu.edu.cn}

%	\author{Lingmin Liao}
%
%	\address{School of Mathematics and Statistics, Wuhan University, Wuhan, Hubei 430072, China}
%
%	\email{lmliao@whu.edu.cn}


	\begin{abstract}
	The mass transference
	principle, discovered by Beresnevich and Velani [Ann. of Math. (2), 2006], is a landmark result in Diophantine approximation that allows us to obtain the Hausdorff measure theory of $\limsup$ set. Another important tool is the notion of large intersection property, introduced and systematically studied by Falconer [J. Lond. Math. Soc. (2), 1994]. The former mainly focuses on passing between full (Lebesgue) measure and full Hausdorff measure statements, while the latter transfers full Hausdorff content statement to Hausdorff dimension. From this perspective, the proofs of the two results are similar but often treated in different ways.

	In this paper, we establish a general mass transference principle from the viewpoint of Hausdorff content, aiming to provide a unified proof for the aforementioned results. More precisely, this principle enables us to transfer the Hausdorff content bounds of a sequence of open sets $E_n$ to the full Hausdorff measure statement and large intersection property for $\limsup E_n$.  One of the advantages of our approach is that the verification of the Hausdorff content bound does not require the construction of Cantor-like subset, resulting in a much simpler proof. As an application, we provide simpler proofs for several mass transference principles.
%	We also give an easier check condition that implies the large intersection property for related $\limsup$ sets.
	\end{abstract}
	\maketitle
	\setcounter{tocdepth}{1}
   \tableofcontents
\section{Introduction}\label{s:intro}
Throughout, let $(X,\dist)$ be a compact metric space. Let $\{E_n\}_{n\ge 1}$ be a sequence of open subsets in $X$. Define the $\limsup$ set
\[\limsup_{n\to\infty} E_n=\bigcap_{N=1}^\infty\bigcup_{n=N}^\infty E_n.\]
In Diophantine approximation, many sets of interest can be written as a $\limsup$ set defined by a  sequence of proper subsets $\{E_n\}_{n\ge 1}$. A fundamental problem in the theory of metric Diophantine approximation is to determine the `size' of such $\limsup$ set in terms of Lebesgue measure, Hausdorff dimension and Hausdorff measure. When the sequence $\{E_n\}_{n\ge 1}$ is a collection of balls, the mass transference principle (MTP) discovered by Beresnevich and Velani \cite{BV06} is a remarkable result which allows us to pass between Lebesgue measure and Hausdorff measure statements for related $\limsup$ sets. Since then, there has appeared quite a substantial number of works building
and generalizing their result. See for example \cite{AB19,AB18,BV06,HS19,KR21,Pe22,PR15,WW21,WWX15,WZ21,Zh21} and reference therein. For an overview of the research on MTP, we refer to a recent survey article \cite{AD23} by Allen and Daviaud.

In \cite{Fal94}, Falconer introduced the notion of large intersection property which aimed at describing the size of $\limsup$ sets from another point of view. Roughly speaking, for a sequence of $G_\delta$-sets $\{A_n\}_{n\ge 1}$, if each $A_n$ has Hausdorff dimension at least $s$ and large intersection property, then so does their intersection $\bigcap_{n\ge 1} A_n$. Since many $\limsup$ sets of interest are $G_\delta$-sets, Falconer's large intersection property asserts that their intersection is large in the sense of Hausdorff dimension. However, a limitation of this property is that it only provides information about Hausdorff dimension, rather than Hausdorff measure. So in many cases, the proofs of full Hausdorff measure statement (the aim of MTP) and large intersection property for $\limsup$ sets are often established separately, although intuitively one would expect that there is a connection between them.

The main purpose of this paper is to present a unified but simpler proof of these two statements from the perspective of Hausdorff content. Our result can be interpreted as the full Hausdorff content statement implies both the full Hausdorff measure statement and the large intersection property for $\limsup$ sets. As an application, we verify that several well-known $\limsup$ sets, which naturally appear in Diophantine approximation, satisfy the full Hausdorff content statement. It turns out that our approach can give simpler and clearer proofs of the results in \cite{AB19,BV06,GN24,KR21,WW21,WZ21,Zh21}, and the conclusions are stronger. An interesting phenomenon is that MTP with ``local scaling property'' by Allen and Baker \cite{AB19} can be viewed as a special case of MTP ``from balls to open sets'' (See Corollary \ref{c:implied}).
%We also provide an easier check condition that implies the large intersection property.

%The main purpose of this paper is to present a unified proof of these two statements from the perspective of Hausdorff content. Our result can be interpreted as the full Hausdorff content statement implies both the full Hausdorff measure statement and the large intersection property for $\limsup$ sets. As an application, we verify that several well-known $\limsup$ sets, which naturally appear in Diophantine approximation, satisfy the full Hausdorff content statement. An interesting phenomenon is that MTP with ``local scaling property'' \cite{AB19} and MTP from rectangles to rectangles \cite{WW21} can both be viewed as a special case of ``from balls to open sets''. It turns out that our proofs are similar but simpler than those already presented in \cite{AB19,BV06,KR21,WW21,WZ21,Zh21}, and the conclusions are stronger.

We detail a few key papers that are particularly relevant to this work, and outline some important recent results. To begin with, we gather some necessary notation that will be used throughout.

By $a\lesssim b$ we mean there is an unspecified constant $c$ such that $a\le cb$. By $a\asymp b$ we mean $a\lesssim b$ and $b\lesssim a$. An open ball with center $x\in X$ and radius $r>0$ is denoted by $B(x,r)$. Occasionally, the radius of a ball $B$ in $X$ will not be explicitly written, and we use the notation $r(B)$ to denote its radius. For any $t>0$, we denote by $tB$ the ball $B$ scaled by a factor $t$; i.e. $tB:=B(x,tr)$.

Let $f\colon \R^+\to\R^+$ be a {\em dimension function}, i.e.\,$f(r)$ is a continuous, nondecreasing function defined on $\R^+$ and satisfying $f(0)=0$. We say that a function $f\colon \R^+\to\R^+$ is {\em doubling} if there exists a constant $D>1$ such that
\begin{equation}\label{eq:doubling}
	f(2r)<D f(r)\qquad\text{for all $r>0$}.
\end{equation}
 Let $f$ be a dimension function. For a set $A\subset X$ and $\eta>0$, let
\[\mathcal H_\eta^f(A)=\inf\bigg\{\sum_{i}f(|B_i|):A\subset \bigcup_{i\ge 1}B_i, \text{ where $B_i$ are balls with $|B_i|\le \eta$}\bigg\},\]
where $|\cdot|$ denotes the diameter of a set.
The {\em Hausdorff $f$-measure} of $A$ is then defined  as
\[\hm^f(A):=\lim_{\eta\to 0}\mathcal H_\eta^f(A).\]
When $\eta=\infty$, $\hc^f(A)$ is referred to as {\em Hausdorff $f$-content} of $A$. If $f(r)=r^s$, then we write $\hm^s$ and $\hc^s$ in place of $\hm^f$ and $\hc^f$, respectively.

Following \cite{Bu04}, for two dimension functions $f$ and $h$, by $h\preceq f$ we mean that
\[r\mapsto \frac{h(r)}{f(r)}\text{ is monotonic and } \lim_{r\to 0^+}\frac{h(r)}{f(r)}>0.\]
The monotonicity in the definition is crucial in the current work. The limit could be finite or infinite. If the limit is infinite, then we write $h\prec f$ instead of $h\preceq f$.

\begin{defn}[$g$-Ahlfors regular]
	Let $g$ be a dimension function. The Hausdorff measure $\hm ^g$ is said to be $g$-Ahlfors regular if $g$ is doubling and
	\[\hm^{g}\big(B(x,r)\big)\asymp g(r)\]
	for all $x\in X$ and $r>0$.
\end{defn}
\begin{rem}
	For our purpose, $g$ is required to be doubling here. If $g(r)=r^\delta$,  the Hausdorff measure $\hm^g$ is also referred to as $\delta$-Ahlfors regular.
\end{rem}

Throughout this paper, we assume that the metric space $(X,\dist)$ supports a $g$-Ahlfors regular measure $\hm^g$. Normally, this will not be written explicitly unless we take $g(r)=r^\delta$ for some $\delta>0$.

\subsection{MTP from balls to open sets}
The following principle was established by Beresnevich and Velani in \cite[Theorem 2]{BV06}.
\begin{thm}[{\cite[Theorem 2]{BV06}}]\label{t:BV}
	Let $f$ be a dimension function such that $f\preceq g$. Assume that $\{B(x_n,r_n)\}_{n\ge 1}$ is a sequence of balls in $X$ with $r_n\to 0$ as $n\to\infty$. If
	\[\hm^g\Big(\limsup_{n\to\infty}B\big(x_n,g^{-1}f(r_n)\big)\Big)=\hm^g(X),\]
	then,
	\[\hm^f\Big(\limsup_{n\to\infty}B(x_n,r_n)\Big)=\hm^f(X).\]
\end{thm}

	Theorem \ref{t:BV} is extended by Allen and Baker \cite{AB19} to limsup sets defined by neighbourhoods of sets satisfying a certain local scaling property.
	\begin{defn}[$\kappa$-scaling property]
		Given $0\le \kappa<1$ and a sequence of sets $\{\rc_\alpha:\alpha\in J\}$ in $X$ indexed by an infinite, countable set $J$, call $\{\rc_\alpha:\alpha\in J\}$ having $\kappa$-scaling property if for any $0<\eta\le r<|X|$, $\alpha\in J$ and $x\in \rca$, we have
		\[g(\eta)^{1-\kappa}\cdot g(r)^{\kappa}\asymp \hm^g\big(B(x,r)\cap \Delta(\rc_\alpha,\eta)\big),\]
		where $\Delta(\rca,\eta)$ is the $\eta$-neighbourhood of $\rca$ defined by
		\[\Delta(\rca,\eta):=\{x\in X:\dist(x,\rc_\alpha)<\eta\}.\]
	\end{defn}
	\begin{rem}
		The sets $\rc_\alpha$ will be referred to as {\em resonant sets}. Examples of sets satisfying the $\kappa$-scaling property include self-similar sets satisfying the open set condition and smooth compact manifolds embedded in $\R^n$. Roughly speaking, when $g(r)=r^\delta$, the quantities $\delta\kappa$ and $\delta(1-\kappa)$ represent the dimension and `codimension' of the resonant set $\rca$, respectively. We refer the reader to \cite[\S 2]{AB19} for more detailed arguments for why such sets satisfy $\kappa$-scaling property.
	\end{rem}

   Take $J=\N$ and let $
   \Upsilon=\{\Upsilon_n\}_{n\ge 1}$ be a sequence of non-negative reals such that $\Upsilon_n\to 0$ as $n\to\infty$. Consider the set
   \[\Lambda(\Upsilon):=\{x\in X:x\in \Delta(\rc_n,\Upsilon_n)\text{ for i.m.\,$n\in\N$}\},\]
   where i.m. stands for infinitely many.   \begin{thm}[{\cite[Theorem 1]{AB19}}]\label{t:mtpbtor}
   	Let $f\preceq g$ be a dimension function such that $f/g^\kappa$ is also a dimension function. If
   	\[\hm^g\bigg(\Lambda\bigg(g^{-1}\bigg(\bigg(\frac{f(\Upsilon)}{g(\Upsilon)^\kappa}\bigg)^{\frac{1}{1-\kappa}}\bigg)\bigg)\bigg)=\hm^g(X),\]
   	then
   	\[\hm^f\big(\Lambda(\Upsilon)\big)=\hm^f(X).\]
   \end{thm}

   Further improvement was given by Koivusalo and Rams \cite{KR21} and Zhong \cite{Zh21}. The former only dealed with the Hausdorff dimension and the latter complement the Hausdorff measure theory.

   \begin{thm}[{\cite[Theorem 3.1]{KR21}} and {\cite[Theorem 1.6]{Zh21}}]\label{t:BtO}
   	Let $\{B_n\}_{n\ge 1}$ be a sequence of balls in $[0,1]^d$ with radii tending to 0. Let $g(r)=r^d$ and let $f$ be a dimension function such that $f\preceq g$.  Let $\{E_n\}_{n\ge 1}$ be a sequence of open sets satisfying $E_n\subset B_n$.  Assume that $\hm^g(\limsup B_n)=\hm^g([0,1]^d)$, and that there exists $c>0$ such that $\hc^f(E_n)>c\hm^g(B_n)$ for all $n\ge 1$. Then, one has
   	\[\hm^f\Big(\limsup_{n\to\infty}E_n\Big)=\hm^f([0,1]^d).\]
   \end{thm}
   \begin{rem}
   	In the next section, we will state a metric space version of Theorem \ref{t:BtO}, and subsequently show that this general version indeed implies Theorem \ref{t:mtpbtor} (see Section \ref{s:implied} for the proof).
   \end{rem}

   The results presented in this subsection, give us MTP from ``balls to open sets". Generally speaking, if there is a full (Lebesgue) measure statement for a $\limsup$ set defined by balls, then there will be a full Hausdorff measure statement for the $\limsup$ set defined by shrinking the balls to open sets. A shortage of these principles is that the initial $\limsup$ set with full measure is required to be defined by balls. However, many $\limsup$ sets with full measure that appear in the weighted Diophantine approximation are defined by rectangles instead of balls, so these principles do not apply. To deal with this problem, Wang and Wu \cite{WW21} developed the transference principle from rectangles to rectangles.
%   \begin{rem}
%   	The statement which they prove is formulated in terms of a generalised version of the singular value function $\varphi^f$ (See Definition \ref{eq:svfunction}) intially introduced by \cite{Fal88}. We will see in Proposition To be more precise, let $f$ be a dimension function and let $E\subset \R^k$ be a
%   	Borel set. We Define the generalised singular value function of $E$ as
%   	\[\varphi^f(E):=\sup_{\mu}\inf_{x\in E}\inf_{r>0} \frac{f(r)}{\mu\big(B(x,r)\big)},\]
%   	where the supremum is taken over Borel probability measures supported on $E$.
%   \end{rem}
\subsection{MTP from rectangles to rectangles}
% Motivated by the fact that many $\limsup$ sets appearing in weighted Diophantine approximation are defined by rectangles, Wang and Wu \cite{WW21} established a rather complete theory in this regard, which is called MTP from rectangles to rectangles.

Fix an integer $d\ge 1$. For each $1\le i\le d$, let $(X_i,\dist_i,m_i)$ be a compact metric space with $m_i$ a $\delta_i$-Ahlfors regular measure. Consider the product space $(X^\times,\dist^\times,m^\times)$, where
\begin{equation}\label{eq:specialmetri}
	X^\times=\prod_{i=1}^{d}X_i,\quad m^\times=\prod_{i=1}^dm_i,\quad \dist^\times=\max_{1\le i\le d}\dist_i.
\end{equation}
We add the symbol $\times$ in the superscript  to emphasize the product property. So a ball $B(x,r)$ in $X^\times$ is in fact the product of balls in $\{X_i\}_{1\le i\le d}$, i.e.
\[B(x,r)=\prod_{i=1}^dB(x_i,r),\quad\text{for $x=(x_1,\dots, x_d)\in X^\times$}.\]

Let $J$ be an infinite countable index set. Let  $\beta\colon J\to \R^+:\alpha\mapsto\beta(\alpha)$ be a positive function such that for any $M>1$, $\{\alpha\in J:\beta(\alpha)<M\}$ is finite. Let $\{l_n\}_{n\ge 1}$ and $\{u_n\}_{n\ge 1}$ be two sequences such that
\[l_n<u_n\qaq\lim_{n\to\infty}l_n=\infty.\]
Define
\begin{equation}\label{eq:Jn}
	J_n=\{\alpha\in J:l_n< \beta(\alpha)\le u_n\}.
\end{equation}
By the assumption on the function $\beta$, the cardinality of $J_n$ is finite. Let $\rho\colon\R^+\to\R^+$ be non-increasing and $\rho(u)\to 0$ as $u\to\infty$.

For each $1\le i\le d$, let $\{\rci:\alpha\in J\}$ be a sequence of subsets of $X_i$.
% satisfying the $\kappa_i$-scaling property.
 Define
\[\bigg\{\rcat=\prod_{i=1}^d\rci:\alpha\in J\bigg\}.\]
For any $\alpha\in J$, $r>0$ and $\bm a=(a_1,\dots, a_d)\in (\R^+)^d$, denote
\[\Delta(\rcat,r^{\bm a})=\prod_{i=1}^d\Delta(\rci,r^{a_i}),\]
where $\Delta(\rc_{\alpha,i},r^{a_i})$ is the neighborhood of $\rci$ in $X_i$ and call it the part of $\Delta(\rcat,r^{\bm a})$ in the $i$th direction.

Fix $\bm a=(a_1,\dots, a_d)\in (\R^+)^d$. For $ \bm t=(t_1,\dots,t_d)\in(\R^+)^d$, define
%\[W(\bm t)=\{x\in X:x\in \Delta(\rc_\alpha,\rho(\beta_\alpha)^{\bm a+\bm t})\text{ for i.m.\,$\alpha\in J$}\},\]
%and more generally, replacing $\rho^{\bm t}$ by general functions $\Psi=(\psi_1,\dots,\psi_d)\colon\R^+\to (\R^+)^d$, the set
\[W(\bm t)=\big\{x\in X^\times:x\in\Delta\big(\rcat,\rho(u_n)^{\bm a+\bm t}\big),\text{ $\alpha\in J_n$, for i.m.\,$n$}\big\},\]
where $u_n$ is corresponding to $\alpha\in J_n$ given in \eqref{eq:Jn}.
%The smaller `rectangle' $\Delta(\rcat,\rho(u_n)^{\bm a+\bm t})$ is regarded as the shrunk one from $\Delta(\rcat,\rho(u_n)^{\bm a})$.
More generally, one can replace $\bm t$ by a sequence $\{\bm t_n\}_{n\ge 1}$ of vectors in $(\R^+)^d$ with $ \bm t_n=(t_{n,1},\dots,t_{n,d})$, and consider the set
\[W\big(\{\bm t_n\}\big)=\big\{x\in X^\times:x\in\Delta\big(\rcat,\rho(u_n)^{\bm a+\bm t_n}\big),\text{ $\alpha\in J_n$, for i.m.\,$n$}\big\}.\]

Inspired by the notions of local ubiquitous systems defined by Beresnevich, Dickinson and Velani \cite{BDV06}, Wang and Wu \cite{WW21} defined the natural corresponding notions of local ubiquity for rectangles.
%Throughout we assume the same
%setting and use the same notation as described in the previous subsection.
\begin{defn}[Local ubiquity system for rectangles]
	Call $(\{\rcat\}_{\alpha\in J},\beta)$ a local ubiquity system for rectangles with respect to $(\rho,\bm a)$ if there exists a constant $c>0$ such that for any ball $B$ in $X^\times$,
	\[\limsup_{n\to\infty} m^\times\bigg(B\cap\bigcup_{\alpha\in J_n}\Delta \big(\rcat, \rho(u_n)^{\bm a}\big)\bigg)>cm^\times(B).\]
\end{defn}
\begin{defn}[Uniform local ubiquity system for rectangles]
	Call $(\{\rcat\}_{\alpha\in J},\beta)$ a uniform local ubiquity system for rectangles with respect to $(\rho,\bm a)$ if there exists a constant $c>0$ such that for any ball $B$ in $X^\times$,
	\[ \liminf_{n\to\infty}m^\times\bigg(B\cap\bigcup_{\alpha\in J_n}\Delta \big(\rcat, \rho(u_n)^{\bm a}\big)\bigg)>cm^\times(B).\]
\end{defn}
Assuming local ubiquity for rectangles, Wang and Wu  \cite{WW21} established the following MTP from rectangles to rectangles.
\begin{thm}[{\cite[Theorem 3.2]{WW21}}]\label{t:mRtoR}
	For $1\le i\le d$, assume the $\kappa_i$-scaling property for $\{\rci:\alpha\in J\}$. Assume the local ubiquity for rectangles. One has
	\[\hm^{s(\bm t)}\big(W(\bm t)\big)=\hm^s(X^\times),\]
	where
	\begin{equation*}
		\begin{split}
			s(\bm t):=\min_{\tau\in\ca}\bigg\{\sum_{i\in\ck_1(\tau)}\delta_i&+\sum_{i\in\ck_2(\tau)}\delta_i+\sum_{i\in\ck_3(\tau)}\delta_i\kappa_i\\
			&+\frac{\sum_{i\in\ck_3(\tau)}\delta_ia_i(1-\kappa_i)-\sum_{i\in\ck_2(\tau)}\delta_it_i(1-\kappa_i)}{\tau}\bigg\},
		\end{split}
	\end{equation*}
	and where
	\[\ca=\{a_i,a_i+t_i:1\le i\le d\}\]
	and for each $\tau\in\ca$, the sets $\ck_1(\tau),\ck_2(\tau)$ and $\ck_3(\tau)$ are defined as
	\[\ck_1(\tau):=\{i:a_i\ge \tau\},\quad\ck_2(\tau)=\{i:a_i+t_i\le \tau\}\setminus \ck_1(\tau),\]
	and
	\[\ck_3(\tau)=\{1,\dots,d\}\setminus\big(\ck_1(\tau)\cup \ck_2(\tau)\big).\]
\end{thm}
For the geometric explanation of the dimensional number $s(\bm t)$ and the sets $\ck_i(\tau)$, we refer to \cite[\S 4]{WW21}.

Wang and Wu \cite{WW21} also gave the Hausdorff dimension of $W(\bm t)$ by weakening the local ubiquity assumption to a full measure assumption.
\begin{thm}[{\cite[Theorem 3.4]{WW21}}]\label{t:underfull}
	For $1\le i\le d$, assume the $\kappa_i$-scaling property for $\{\rci:\alpha\in J\}$ and
	\[m^\times\bigg(\limsup_{\alpha\in J,\beta(\alpha)\to \infty}\Delta\big(\rcat,\rho(\beta(\alpha))^{\bm a}\big)\bigg)=m^\times(X^\times).\]
	Then, one has
	\[\hdim W(\bm t)\ge s(\bm t),\]
	where $\hdim$ denotes the Hausdorff dimension.
\end{thm}

The Hausdorff dimension of general case $W\big(\{\bm t_n\}\big)$ is obtained under the uniform local ubiquity.
\begin{thm}[{\cite[Theorem 3.3]{WW21}}]\label{t:generalmRtoR}
	For $1\le i\le d$, assume the $\kappa_i$-scaling property for $\{\rci:\alpha\in J\}$. Assume the uniform local ubiquity for rectangles. One has
	\[\hdim W\big(\{\bm t_n\}\big)\ge \max_{\bm t\in \cu}s(\bm t),\]
	where $\cu$ is the set of accumulation points of the sequence $\{\bm t_n\}_{n\ge 1}$.
\end{thm}

\subsection{Dynamical dimension transference principle}
Compared with the classic Diophantine approximation which concerns the quantitative properties
of the distribution of rational numbers or resonant sets, dynamical Diophantine approximation concerns that of the orbits in a dynamical system.
%Motivated by the classical Poincar\'e's recurrence theorem and Birkhoff's ergodic theorem, many dynamically defined sets of interest are also $\limsup$ sets.
In analogy with MTP, Wang and Zhang \cite{WZ21} provided a general principle for the Hausdorff dimension of the
$\limsup$ set arising in a general expanding dynamical system.

Let us now specify their setting. Let $(X,\dist,T)$ be a topological dynamical system with $X$ a compact metric space and $T\colon X\to X$ a continuous transformation. Suppose that $(X,\dist,T)$ satisfies the following conditions.

\noindent \textbf{Hypothesis A: dynamical ubiquity.} Fix $y_0\in X$ and denote
\[Y=\bigcup_{n\ge 0} T^{-n}y_0.\]
There exist a strictly positive continuous function $\phi\colon X\to \R^+$ and constants $0<c_2<1<c_1<\infty$ such that
\begin{enumerate}[({A}1)]
	\item Covering property: given $y\in Y$ and $n\in\N$, the following family of balls covers $X$
	\[\bigl\{B(z,c_1e^{-S_n\phi(z)}):z\in T^{-n}y\bigr\},\]
	where $S_n\phi(z)=\sum_{i=0}^{n-1}\phi(T^iz)$ denotes the ergodic sum.
	\item Separation property: given $y\in Y$ and $n\in \N$, the following family of balls are disjoint
	\[\bigl\{B(z,c_2e^{-S_n\phi(z)}):z\in T^{-n}y\bigr\}.\]
\end{enumerate}

\noindent \textbf{Hypothesis B: local conformality.} For any $\lambda>1$, there exists $0<b_\lambda<1$ such that, whenever $0<b\le b_\lambda$, it holds that
\[B(Tz,\lambda^{-1}b)\subset T\big(B(z,e^{-\phi(z)}b)\big)\subset B(Tz,\lambda b)\]

\noindent \textbf{Hypothesis C: exactness.} Given any ball $B$, there exists $N\in\N$ such that $T^nB=X$ whenever $n\ge N$.

Let $\psi\colon X\to\R^+$ be a strictly positive continuous function. Let $y_0\in X$ and $\phi$ be as in hypothesis (A). Consider the following ``well-approximable" set driven by the dynamical system $(X,\dist,T)$,
\[\cw(T,\psi):=\big\{x\in X:\dist(x,z)<e^{-S_n(\phi+\psi)(z)}\text{ for some $z\in T^{-n}y_0$, i.m.\,$n\in\N$}\big\}.\]
\begin{thm}[{\cite[Theorem 1.4]{WZ21}}]\label{t:dmmtp}
	Assume the hypotheses (A)--(C). Let $\psi$ be a non-negative continuous function over $X$. Then,
	\[\hdim \cw(T,\psi)=\inf\big\{t\ge 0:P\big(-t(\phi+\psi)\big)\le 0\big\},\]
	where $P\big(-t(\phi+\psi)\big)$ is the pressure function defined as follows: For any $y\in Y$
	\[P\big(-t(\phi+\psi)\big):=\lim_{n\to\infty}\frac{1}{n}\log\sum_{z\in T^{-n}y_0}e^{-tS_n(\phi+\psi)}.\]
\end{thm}
%The proof of the existence of the limit is given in \cite[Theorem 2.6]{WZ21}.
%\subsection*{Organization of the paper}
%The rest of the paper is arranged as follows. In Section \ref{s:mian} we present our main results  and their applications. In Section \ref{s:auxiliary result}, we establish some auxiliary results, including a Frostman-type lemma. In Section \ref{s:basic}, we use these auxiliary results to prove Theorem \ref{t:MTPBtO}. In the subsequent sections, we apply our main results to show that the $\limsup$ sets presented in Section \ref{s:intro} have full Hausdorff measure/dimension and large intersection property.  The proofs involve estimating the lower bound on the Hausdorff content of the sets under consideration.

\section{Main results}\label{s:mian}
In this section, we describe the notion of large intersection property and state the main results of the paper.
\subsection{Large intersection property}
The notion of large intersection property was introduced by Falconer \cite{Fal94}, and we refer to that paper for some background. Recall that $X$ is compact and supports a $g$-Ahlfors regular measure.
\begin{defn}\label{d:LIP1}
	Let $0<s\le \hdim X$. We define $\scg_1^s(X)$ to be the class of $G_\delta$-subsets $A$ of $X$ such that there exists a constant $c>0$ such that for any $0<t<s$ and any ball $B$,
	\[\hc^t(A\cap B)>c\hc^t(B).\]
\end{defn}
%We remark that in the Euclidean space, $\cq$ can be taken as the most common dyadic cubes.
The more general notion was given by Bugeaud \cite{Bu04}.

\begin{defn}[General dimension function]\label{d:LIP2}
	Let $f$ be a dimension function. We define $\scg_2^f(X)$ to be the class of $G_\delta$-subsets $A$ of $X$ such that there exists a constant $c>0$ such that for any dimension function $h\prec f$ and any ball $B$,
	\[\hc^h(A\cap B)>c\hc^h(B).\]
\end{defn}

%The original Since $X$ supports a $g$-Ahlfors regular measure, it is not diffcult to derive from \cite{Fal94} and \cite{Bu04} that both $\scg_1^s(X)$ and $\scg_2^f(X)$ are closed under countable intersections and under bi-Lipschitz maps.
%\begin{thm}[{\cite[Theorem B]{Fal94}\footnote{The proof contains a minor mistake and was later corrected by \cite{Bu04}.} and \cite[Theorem 6]{Bu04}}]\label{t:oriLIP}
%	Let $\{A_n\}_{n\ge 1}$ be a sequence of $G_\delta$-sets. If $A_n\in \scg_1^s(X)$ for every $n\ge 1$, then $\bigcap_{n=1}^\infty A_n\in\scg_1^s(X)$. The same argument holds if $\scg_1^s(X)$ is replaced by $\scg_2^f(X)$.
%\end{thm}
As stated in Section \ref{s:intro}, having large intersection property does not imply the full Hausdorff measure statement. Because, in the definition of $\scg_1^{s}(X)$, we do not require that the Hausdorff content bound holds for $t=s$. The same reason for the class $\scg_2^f(X)$. A reasonable way to overcome this shortfall is to modify the definitions as follows.
%So, it would be much reasonable to develop a theory that both implies full Hausdorff measure statement and the large intersection property.
\begin{defn}\label{d:LIP}
	Let $f$ be a dimension function. We define $\scg^f(X)$ to be the class of $G_\delta$-subsets $A$ of $X$ such that there exists a constant $c>0$ such that for any ball $B$,
	\[\hc^f(A\cap B)>c\hc^f(B).\]
\end{defn}
\begin{rem}
	From the condition, it is not difficult to check that the Hausdorff content bound holds for all dimension functions $h\preceq f$.
\end{rem}
When $f(r)=r^s$, we write $\scg^s(X)$ in place of $\scg^f(X)$. With this new definition given above, we are able to develop a theory that implies both the full Hausdorff measure statement and the large intersection property, which will be presented in the next section.
\subsection{Statement of results}
Our first result shows that the new class $\scg^f(X)$ has all the desired properties.
\begin{thm}\label{t:LIPfull}
		Let $f$ be a dimension function such that $f\preceq g$. The class $\scg^f(X)$ is closed under countable intersection. Moreover, for any $A\in\scg^f(X)$, we have
	\[\hm^f(A)=\hm^f(X).\]
\end{thm}

This gives a way to prove simultaneously the full Hausdorff measure statement and the large intersection property for $\limsup$ sets from the perspective of Hausdorff content. On verifying the Hausdorff content bound, we obtain a stronger version of Theorem \ref{t:BtO}.

\begin{thm}\label{t:MTPBtO}
	Let $f$ be a dimension function such that $f\preceq g$. Assume that $\{B_n\}_{n\ge 1}$ is a sequence of balls in $X$ with radii tending to 0, and that $\hm^g(\limsup B_n)=\hm^g(X)$. Let $\{E_n\}_{n\ge 1}$ be a sequence of open sets satisfying $E_n\subset B_n$.  Assume that there exists a constant $c>0$ such that
	\begin{equation}\label{eq:contentbound}
		\hc^f(E_n)>c\hm^g(B_n)\quad\text{ for all $n\ge 1$}.
	\end{equation}
	 Then,
	\[\limsup_{n\to\infty}E_n\in \scg^f(X).\]
\end{thm}
\begin{rem}
	Ghosh and Nandi \cite[Lemma 3.13]{GN24} established an analogous result within the setting of $\scg_1^s(X)$. As previously explained, this result provides insights into the Hausdorff dimension rather than the Hausdorff measure for relevant $\limsup$ set. In \cite[Theorem 1.2]{Er22}, Eriksson-Bique obtained a lower bound on the Hausdorff content of $\limsup E_n$ under less restrictive conditions, but did not address the full Hausdorff measure statement or the large intersection property. Here we impose stronger conditions in order to give a more complete description of the $\limsup$ set.
%	Our proof is classic, but somewhat different from those of \cite[Lemma 3.13]{GN24}, \cite[Theorem 1.6]{Zh21} and \cite[Theorem 1.2]{Er22}.
%	In the spirit of Theorem \ref{t:generalmtp}, it is not difficult to obtain \eqref{eq:Hausmea} from \eqref{eq:contlow}. In order to keep the article to amanageable length, we will not give the proof of this theorem.
\end{rem}
%\begin{rem}
%	The idea is to prove that $F_n=\bigcup_{k\ge n}E_k$ satisfies the condition in Definition \ref{d:LIP}, and so $F_n\in \scg^f(X)$. This also yields the desired conclusion since $\scg^f(X)$ is closed under countable intersection and
%	\[\limsup_{n\to \infty}E_n=\bigcap_{n=1}^\infty F_n.\]
%\end{rem}

A somewhat surprising consequence of Theorem \ref{t:MTPBtO} is that the conditions for the large intersection property can be significantly weakened.
\begin{cor}\label{c:weaken}
	Let $f$ be a dimension function such that $f\preceq g$. Assume that $\{B_n\}_{n\ge 1}$ is a sequence of balls in $X$ with radii tending to 0, and that $\hm^g(\limsup B_n)=\hm^g(X)$. Let $\{E_n\}_{n\ge 1}$ be a sequence of open sets (not necessarily contained in $B_n$). The following statements hold.
	\begin{enumerate}[(1)]
		\item  If there exists a constant $c>0$ such that for any $B_k$,
		\[\limsup_{n\to\infty} \hc^f (E_n\cap  B_k)>c\hm^g(B_k),\]
%		\begin{equation}\label{eq:cond}
%			\limsup_{n\to\infty} \hc^f (E_n\cap  B_k)>c\hm^g(B_k),
%		\end{equation}
		then
		\[\limsup_{n\to\infty} E_n\in\scg^f(X).\]
		\item If for any $h\prec f$, there exists a constant $c>0$ such that for any $B_k$,
%		\begin{equation}\label{eq:moreeasycheck}
%			\limsup_{n\to\infty} \hc^h (E_n\cap  B_k)>c\hm^g(B_k),
%		\end{equation}
\[\limsup_{n\to\infty} \hc^h (E_n\cap  B_k)>c\hm^g(B_k),\]
		then
		\[\limsup_{n\to\infty} E_n\in\scg_2^f(X).\]
		The same argument holds for $\scg_1^s(X)$ if $f(r)=r^s$ and $h$ is replaced by $t<s$.
	\end{enumerate}
\end{cor}
\begin{rem}
	The condition in item (1) is stronger than that in item (2), and therefore the conclusion is stronger. In general, one can let $F_n=\bigcup_{k\ge n}E_k$ and then verify that $\{F_n\}_{n\ge 1}$ satisfies the condition in item (1) or (2). This also yields the desired conclusion since
	\[\limsup_{n\to \infty}E_n=\limsup_{n\to \infty}F_n.\]
\end{rem}
\begin{rem}
	The advantages of Theorem \ref{t:MTPBtO} and Corollary \ref{c:weaken} are that all we need to do is to verify Hausdorff content bound condition, instead of constructing a Cantor-like subset of $\limsup E_n$. Although the verification is similar to the proofs in \cite{AB19,BV06,KR21,WW21,Zh21}, it is indeed much simpler in comparison, see Section \ref{s:rtor} for example.
\end{rem}

%A somewhat surprising consequence of Theorem \ref{t:fullcontent} is that the conditions for the large intersection property can be significantly weakened.
%\begin{cor}\label{c:weaken}
%	Let $\{B_k\}_{k\ge 1}$ be a sequence of balls such that $\hm^g(\limsup B_k)=\hm^g(X)$. Let $\{E_n\}_{n\ge 1}$ be a sequence of open sets. Let $f$ be a dimension function such that $f\preceq g$. If for any $h\prec f$, there exists a constant $c>0$ such that for any $B_k$,
%	\begin{equation}\label{eq:moreeasycheck}
%		\limsup_{n\to\infty} \hc^h (E_n\cap  B_k)>c\hm^g(B_k),
%	\end{equation}
%	then
%	\[\limsup_{n\to\infty} E_n\in\scg_2^f(X).\]
%	The same argument holds for $\scg_1^s(X)$ if $h$ is replaced by $t<s$.
%\end{cor}
%\begin{rem}
%	The condition \eqref{eq:moreeasycheck} is easier to check than those in Definitions \ref{d:LIP1} and \ref{d:LIP2}, as it gives us more freedom to choice a suitable measure supported on $E_n\cap B$. In \cite[Lemma 2.1]{PR15}, Persson and Reeve obtained a similar but weaker result  than ours.
%\end{rem}

Applying Theorem \ref{t:MTPBtO}, we obtain a stronger statement of Theorem \ref{t:mtpbtor}.
	\begin{cor}\label{c:implied}
	Theorem \ref{t:mtpbtor} is implied by Theorem \ref{t:MTPBtO}. Moreover,
	\[\Lambda(\Upsilon)\in\scg^f(X).\]
	\end{cor}
	The implication arises from a straightforward observation. While the proof is uncomplicated, to the best of the author's knowledge, it has not been previously presented in the literature.

	Applying Corollary \ref{c:weaken}, we obtain a stronger statement of Theorem \ref{t:mRtoR}.
	\begin{thm}\label{t:meaRtoR}
		Under the same setting of Theorem \ref{t:mRtoR}, we have
		\[W(\bm t)\in\scg^{s(\bm t)}(X^\times).\]
	\end{thm}

%	\begin{rem}\label{r:rtor}
%		Recall the product space $(X^\times,\dist^\times,m^\times)$ in \eqref{eq:specialmetri}.
%		Let us detailed why MTP from rectangles to rectangles (Theorem \ref{t:mRtoR}) can be viewed as a special version of MTP from balls to open sets (Theorem \ref{t:MTPBtO}). Since $X^\times$ is compact, one can find a countable sequence $\{B_i\}_{i\ge 1}$ of balls such that $\limsup B_i=X^\times$. For each $i$, apply the local ubiquity condtion we have
%		\[m^\times\bigg(B_i\cap\bigcup_{\alpha\in J_n}\Delta (\rcat, \rho(u_n)^{\bm a})\bigg)\asymp m^\times(B_i).\]
%		Let $F_i=B_i\cap\big(\cup_{\alpha\in J_n}\Delta (\rcat, \rho(u_n)^{\bm a+\bm t})\big)$. By Theorem \ref{t:MTPBtO}, the proof is reduced to show that
%		\[\hc^{s(\bm t)}(F_i)\apprge m^\times(B_i).\]
%%		This can be done by showing that for any ball $B$,
%%		\[\limsup_{n\to\infty}\hc^{s(\bm t)}\bigg(B\cap\bigcup_{\alpha\in J_n}\Delta (\rcat, \rho(u_n)^{\bm a+\bm t})\bigg)\apprge m^\times(B).\]
%	\end{rem}
	Corollary \ref{c:weaken} further enables us to improve the lower bound for $\hdim W\big(\{\bm t_n\}\big)$.
		\begin{thm}\label{t:Hausmeartor}
		Under the same setting of Theorem \ref{t:generalmRtoR}, we have
		\[W\big(\{\bm t_n\}\big)\in \scg_1^s(X^\times),\]
		where $s=\limsup s(\bm t_n)$. If $s$ is attained along a nonincreasing subsequence, then we further have
		\[W\big(\{\bm t_n\}\big)\in \scg^s(X^\times).\]
	\end{thm}
	\begin{rem}
		By \cite[Proposition 3.1]{WW21}, one can indeed prove that $s(\bm t)$ as a function of $\bm t$ is continuous on $(\R^+)^d$. From this it follows that $\limsup s(\bm t_n)\ge \max\{s(\bm t):\bm t\in\cu\}$, where recall that $\cu$ is the accumulation points of the sequence $\{\bm t_n\}_{n\ge 1}$. The equality holds if $\{\bm t_n\}_{n\ge 1}$ is a bounded set, since any sequence attaining the $\limsup s(\bm t_n)$ has a convengent subsequence. However, it is possible that $\cu$ is empty but $\limsup s(\bm t_n)>0$, and hence the inequality is strict. For instance, see the example in the remark below.
%		For example, consider the case $d=2$. Suppose that $a_1=a_2=\delta_1=\delta_2=1$ and $\kappa_1=\kappa_2=0$. Let $t_{n,1}=1$ and $t_{n,2}=n$. Clearly, $\cu$ is empty since $t_{n,2}\to \infty$ as $n\to\infty$, but $s(\bm t_n)=1$ for all $n\ge 1$. In this case, Theorem \ref{t:generalmRtoR} gives no information about the size of $W\big(\{\bm t_n\}\big)$, while the above theorem asserts that this set has full Hausdorff measure and large intersection property.
	\end{rem}
	\begin{rem}
		Theorem \ref{t:Hausmeartor} is also applicable to the ``unbounded case'' presented in \cite[\S  5.3]{LLVZ22}. In that paper, Li, Liao, Velani and Zorin claimed that the Hausdorff dimension of the following set
		\begin{equation}\label{eq:exmset}
			\{(x_1,x_2)\in[0,1)^2:\|2^nx_1\|<e^{-nt}\text{ and }\|3^nx_2\|<e^{-n^2}\text{ for i.m. $n\in\N$}\}
		\end{equation}
		is
		\begin{equation}\label{eq:exmsetdim}
			\min\bigg\{1,\frac{\log2+\log 3}{\log 2+t}\bigg\},
		\end{equation}
		where $\|\cdot\|$ denotes the distance to the nearest integer.
		However, the authors put the proof in their forthcoming paper, which, to the best of our knowledge, is not yet completed. We remark that Theorem \ref{t:Hausmeartor} is applicable to this set. To apply the theorem, one can take $\rho(u_n)=2^{-n}$, $a_1=1$, $a_2=\log 3/\log 2$, $\delta_1=\delta_2=1$ and $\kappa_1=\kappa_2=0$. Let $\bm t_n=(t/\log 2,n/\log 2)$, where $t$ is given above. By a direct computation, we have $\cu$ is empty but for all large $n$,
		\[s(\bm t_n)=\min\bigg\{1+\frac{\log 3-t}{\log 3+n},\frac{\log 2+\log 3}{\log 2+t}\bigg\}.\]
		Letting $n\to\infty$, we can obtain the lower bound on the Hausdorff dimension of this set, which is equal to \eqref{eq:exmsetdim}. Moreover, since $s(\bm t_n)$ is nonincreasing with respect to $n$, we can further obtain the full Hausdorff measure statement for this set, which is not claimed in \cite[\S 5.3]{LLVZ22}.
	\end{rem}

	A stronger version of Theorem \ref{t:underfull} can also be otained.
%	Note that we cannot obtain the Hausdorff measure under the full measure statement.
	\begin{thm}
%		Assume that for each $1\le i\le d$, the measure $m_i$ is $\delta_i$-Ahlfors regular and $\rci$ has the $\kappa_i$-scaling property for $\alpha\in J$. Suppose
%		\[m^\times\bigg(\limsup_{\alpha\in J,\beta(\alpha)\to \infty}\Delta\big(\rcat,\rho(\beta(\alpha))^{\bm a}\big)\bigg)=m^\times(X^\times).\]
		Under the same setting of  Theorem \ref{t:underfull}, we have
		\[W(\bm t)\in\scg_1^{s(\bm t)}(X^\times).\]
	\end{thm}
	\begin{rem}
		The strategy is to verify that for any $0<t<s(\bm t)$,
		\[\lim_{n\to\infty}\hc^t\bigg(\bigcup_{\alpha\in J,\beta(\alpha)\ge n}\Delta\big(\rcat,\rho(\beta(\alpha))^{\bm a+\bm t}\big)\cap B\bigg)\apprge m^\times(B) \]
		holds for all ball $B$. The above inequality can be proved in much the same way as \cite[Theorem 3.4]{WW21}. For modifications that need to be made, one can refer to the proof of Theorem \ref{t:MTPRtO} below. In order to keep the article to a manageable length, we will omit these details.
	\end{rem}

	In analogy to Theorem \ref{t:BtO}, we obtain a MTP from rectangles to small open sets.
\begin{thm}\label{t:MTPRtO}
		Let $(X^\times,\dist^\times,m^\times)$ be as in \eqref{eq:specialmetri}.  Let $\bm a=(a_1,\dots,a_d)\in(\R^+)^d$ with $a_1\le \cdots\le a_d$. Let $\{R_n\}_{n\ge 1}$ be a sequence of rectangles with $R_n=\prod_{i=1}^{d}B(z_{n,i},r_n^{a_i})$ and $r_n\to 0$. Let $\{E_n\}_{n\ge 1}$ be a sequence of open sets with $E_n\subset R_n$ and $|E_n|\le r_n^{a_d}$. Assume that
		\[m^\times\Big(\limsup R_n\Big)=m^\times(X^\times),\]
		and that	there exist $c>0$ and $s>0$ such that
	\[\hc^s(E_n)>cm^\times(R_n)\quad \text{for all $n\ge 1$}.\]
	Then,
	\[\limsup_{n\to\infty}E_n\in\scg_1^s(X^\times).\]
\end{thm}
\begin{rem}

	Unlike the $\limsup$ set defined by shrinking a sequence of balls (Theorem \ref{t:BtO}), for the $\limsup$ set defined by shrinking a sequence of rectangles, one needs to find an optimal cover for the collection $\{E_n\}_{n\ge 1}$ of open sets  instead of one. We refer to \cite[\S 4]{WW21} for more details. So we impose the condition $|E_n|\le r_n^{a_d}$ so that only the optimal cover of $E_n$ itself needs to be considered. However, even with such a condition, we cannot establish the Hausdorff measure theory under the full measure assumption because the `rectangle' setup is more complicated.
%	If such condition is repalced by local ubiquity, it can be showed that the corresponding set has full Hausdorff measure and large intersection property (in the sense of $\scg_2^f(X)$). In order to keep the article to a manageable length, we will not give the proof under the assumption of local ubiquity.
\end{rem}

We also prove the large intersection property for dynamically defined $\limsup$ set under an additional $\delta$-Ahlfors regular assumption.
\begin{thm}\label{t:dmmtp2}
	Under the same setting of Theorem \ref{t:dmmtp} and assume that $X$ supports a $\delta$-Ahlfors regular measure $m$, we have
	\[\cw(T,\psi)\in \scg_1^s(X),\]
	where $s$ denotes the Hausdorff dimension of $\cw(T,\psi)$.
\end{thm}
\begin{rem}
	The additional $\delta$-Ahlfors regular assumption is necessary for us because the theory of large intersection property is built under this assumption. Although a weaker assumption $g$-Ahlfors regularity could be substituted, for the sake of simplicity we will not pursue this generalization.
\end{rem}
\section{Properties of the class $\scg^f(X)$}\label{s:auxiliary result}
In this section, we shall prove that the class $\scg^f(X)$ is closed under countable intersection, and that each set in $\scg^f(X)$ has full Hausdorff measure. The main tool is the net measure, which is defined in terms of coverings by generalised cubes in $X$.
%
%The strategy of the proof of Theorem \ref{t:MTPBtO} is as follows: We show that the condition in Theorem \ref{t:MTPBtO} implies that for any ball $B$,
%	\begin{equation}\label{eq:cond1}
%	\lim_{n\to\infty} \hc^f \bigg(\bigcup_{k=n}^\infty E_k\cap  B\bigg)\apprge \hc^f(B).
%\end{equation}
%This yields the desired conclusion by applying the following lemma.
%\begin{lem}\label{l:generalmtp}
%	Let $\{E_n\}_{n\ge 1}$ be a sequence of open sets. Let $f$ be a dimension function such that $f\preceq g$. If there exists a constant $c>0$ such that for any ball $B$,
%	\[\lim_{n\to\infty} \hc^f \bigg(\bigcup_{k=n}^\infty E_k\cap  B\bigg)>c\hc^f(B),\]
%	then
%	\begin{equation}\label{eq:full2}
%		\hm^f\Big(\limsup_{n\to\infty}E_n\Big)=\hm^f(X)\qaq \limsup_{n\to\infty} E_n\in\scg_2^f(X).
%	\end{equation}
%\end{lem}
%The proof of Theorem \ref{t:MTPBtO} is arranged as follows. In Section \ref{ss:net measure}, we define the net measure and use it to prove Lemma \ref{l:generalmtp}. In Section \ref{s:weighed Hausdorff}, we establish some auxiliary results that will be used to the proof of \eqref{eq:cond1}. The final step, verifying \eqref{eq:cond1}, is reserved to Section \ref{s:basic}.
\subsection{Net measure and large intersection property}\label{ss:net measure}
Recall that $X$ supports a $g$-regular measure.  Clearly, by a volume argument this implies that  $X$ has the {\em finite doubling property}, i.e.\,any ball $B(x,2r)\subset X$ can be covered by finitely many balls of radius $r$.

\begin{thm}[{\cite[Theorem 2.1]{KRS12}}]\label{t:gdc}
	If $X$ is a metric space with the finite doubling property and $0<r<1/3$, then there exists a collection $\cq=\{Q_{n,i}:n\in\Z, i\in \cq_n\subset \N\}$ of Borel sets having the following properties:
	\begin{enumerate}
		\item $X=\bigcup_{i\in \cq_n}Q_{n,i}$ for every $n\in \Z$,
		\item $Q_{n,i}\cap Q_{k,j}=\emptyset$ or $Q_{n,i}\subset Q_{k,j}$ where $n,k\in\Z$, $n\ge k$, $i\in \cq_n$ and $j\in \cq_k$,
		\item for every $n\in \Z$ and $i\in \cq_n$ there exists a point $x_{n,i}\in X$ so that
		\[B(x_{n,i},cr^n)\subset Q_{n,i}\subset \overline{B(x_{n,i},Cr^n)},\]
		where $c=\frac{1}{2}-\frac{r}{1-r}$ and $C=\frac{1}{1-r}$,
		\item $\{x_{n,i}:i\in \cq_n\}\subset \{x_{n+1,i}:i\in \cq_{n+1}\}$ for all $n\in\Z$.
	\end{enumerate}
\end{thm}
Throughout, we will always take $r=1/4$. This means that in the above theorem, $c=1/6> 1/4^2$ and $C=4/3<4$, and therefore
\begin{equation}\label{eq:BsubQsubB}
	B(x_{n,i},4^{-n-2})\subset Q_{n,i}\subset B(x_{n,i},4^{-n+1}).
\end{equation}
 Occasionally, when referring to a	 generalised cube we will use $Q$ or $Q_i$, without mentioning which $\cq_n$ it comes from. Let $f$ be a dimension function. For a set $A$, define
\[\nc^f(A):=\inf\bigg\{\sum_{i}f(|Q_i|):A\subset \bigcup_{i\ge 1}Q_i, \text{ where $Q_i\in\cq$}\bigg\}.\]
Note that $X$ supports a $g$-Ahlfors regular measure. By Theorem \ref{t:gdc} (3) and a volume argument, it is not difficult to verify that for any $n \in\N$ and $x\in X$, the ball $B(x,4^{-n})$ intersects finitely many generalised cubes in $\cq_n$. Therefore, for any $A\subset X$
\begin{equation}\label{eq:eqiv}
	\hc^f(A)\asymp\nc^f(A),
\end{equation}
where the implied constant is absolute.
\begin{rem}\label{r:equiv}
	In the presence of \eqref{eq:BsubQsubB} and \eqref{eq:eqiv}, it can be checked that $A\in\scg^f(X)$ if and only if
	\[\nc^f(A\cap Q)\apprge \nc^f(Q)\quad\text{for all $Q\in\cq$},\]
	where the implied constant does not depend on $f$. The analogous argument holds for $\scg_1^s(X)$ and $\scg_2^f(X)$.
\end{rem}

In general, $\nc^f$ is not a Borel measure and hence many results in measure theory are not applicable to this setting. Nonetheless, one can still obtain a result similar to Lemma \ref{l:density} without using any measure-theoretic result. The proof follows the idea of \cite{Di21} closely, we include it for completeness.
\begin{lem}\label{l:subto=}
	Let $A$ be a Borel subset of $X$. Let $f$ be a dimension function. If there exists a constant $c>0$ such that
	\begin{equation}\label{eq:hclower}
		\nc^f (A\cap Q)>c\nc^f(Q)\quad\text{for all $Q\in \cq$},
	\end{equation}
	then
	\[\nc^f (A\cap Q)=\nc^f(Q)\quad\text{for all $Q\in \cq$}.\]
%	In particular, the same holds for replacing $Q$  by arbitrary open set $U\subset X$.
\end{lem}
\begin{proof}
	Let $Q\subset\cq$ be a generalised cube. Let $\{I_k\}_{k\ge 1}$ be a net cover of $A\cap Q$ such that $I_k\subset Q$ for $k\ge 1$. Since $\nc^f(A\cap Q)\le f(|Q|)<\infty$, we may suppose that $\sum_{k\ge 1}f(|I_k|)<\infty$. For any $\epsilon>0$, there exists $n_0>0$ such that
	\begin{equation}\label{eq:sumsum}
		\sum_{n\ge n_0}\sum_{I_k\in \cq_n}f(|I_k|)<\epsilon.
	\end{equation}
	For such $n_0$, write
	\begin{equation}\label{eq:paetitionQ}
		Q=\bigcup_{Q_l\in\cq_{n_0},\,Q_l\subset Q}Q_l
	\end{equation}
	as a disjoint union of generalised cubes in $\cq_{n_0}$. In view of \eqref{eq:hclower}, each $I_k$ has non-empty intersection with some $Q_l$. Now, we define a finite family $\{D_j\}_{j\ge 1}$ of pairwise disjoint generalised cubes as follows:
	 	\begin{enumerate}[(i)]
		 		\item If $Q_l\subset I_k$ for some $k$, then let $D_j=I_k$,
		 		\item If $I_k\subset Q_l$ for some $k$, then let $D_j=Q_l$.
	  \end{enumerate}
%	\[D_j=\begin{cases}
%		Q_l,&\text{if $I_k\subset Q_l$ for some $k$},\\
%		I_k,&\text{if $Q_l\subset I_k$ for some $j$}.
%	\end{cases}\]
	Obviously, $\{D_j\}_{j\ge 1}$ is a finite net cover of $Q$. Moreover, if $D_j$ is obtained from (ii), then
	\begin{equation}\label{eq:AcapDj}
		A\cap D_j=A\cap Q_l\subset \bigcup_{k: I_k\subset Q_l} I_k.
	\end{equation}
	Namely, the rightmost union forms a cover of $A\cap D_j$. For any $I_k\in\cq_n$ with $n\ge n_0$, by the properties of $\cq$ one has $I_k\subset Q_l$ for some $Q_l\in\cq_{n_0}$. Therefore, by \eqref{eq:AcapDj}
	\[\bigcup_{D_j\in\text{(ii)}}A\cap D_j\subset\bigcup_{n\ge n_0}\bigcup_{I_k\in \cq_n}I_k,\]
	where the notation $D_j\in\text{(ii)}$ means that $D_j$ is obtained from (ii). This implies that
	\[\sum_{n\ge n_0}\sum_{I_k\in \cq_n}f(|I_k|)\ge \sum_{D_j\in\text{(ii)}}\nc^f(A\cap D_j)>c\sum_{D_j\in\text{(ii)}}\nc^f(D_j).\]
	It then follows that
	\[\begin{split}
		\sum_{k=1}^{\infty} f(|I_k|)&=\sum_{n< n_0}\sum_{I_k\in \cq_n}f(|I_k|)+\sum_{n\ge n_0}\sum_{I_k\in \cq_n}f(|I_k|)\\
		&=\sum_{D_j\in\text{(i)}}f(|D_j|)+c^{-1}\sum_{n\ge n_0}\sum_{I_k\in \cq_n}f(|I_k|)+(1-c^{-1})\sum_{n\ge n_0}\sum_{I_k\in \cq_n}f(|I_k|)\\
		&\ge\sum_{D_j\in\text{(i)}}\nc^f(D_j)+\sum_{D_j\in\text{(ii)}}\nc^f(D_j)+(1-c^{-1})\epsilon\\
		&\ge \nc^f(Q)+(1-c^{-1})\epsilon.
	\end{split}\]
	Letting $\epsilon\to 0$, we have $\sum_{k\ge 1} f(|I_k|)\ge \nc^f(Q)$. This is true for all such covers $\{I_k\}_{k\ge 1}$, so
	\[\nc^f (A\cap  Q)=\nc^f(Q).\qedhere\]
\end{proof}

A slightly modification enables us to extend the result to open sets. The proof is also obtained from \cite{Di21}.

\begin{lem}\label{l:subto=2}
	Let $A$ be a Borel subset of $X$. Let $f$ be a dimension function. If there exists a constant $c>0$ such that
	\[\nc^f (A\cap Q)>c\nc^f(Q)\quad\text{for all $Q\in \cq$},\]
	then
	\[\nc^f (A\cap U)=\nc^f(U)\quad\text{for all non-empty open set $U\subset X$}.\]
\end{lem}
\begin{proof}
	By Lemma \ref{l:subto=}, the assumption implies that
		\[\nc^f (A\cap Q)=\nc^f(Q)\quad\text{for all $Q\in \cq$}.\]
	Now, we show the same holds for arbitrary non-empty open set $U$. Since $U$ is non-empty, it can be written as a disjoint union of generalised cubes $\{Q_l\}_{l\ge 1}$, i.e.
	\[U=\bigcup_{l\ge 1}Q_l.\]
	Let $\{I_k\}_{k\ge 1}$ be a net cover of $A\cap U$.  Now, we define a finite family $\{D_j\}_{j\ge 1}$ of pairwise disjoint generalised cubes as follows:
	\begin{enumerate}[(a)]
		\item If $Q_l\subset I_k$ for some $k$, then let $D_j=I_k$,
		\item If $I_k\subset Q_l$ for some $k$, then let $D_j=Q_l$.
	\end{enumerate}
	%	\[D_j=\begin{cases}
		%		Q_l,&\text{if $I_k\subset Q_l$ for some $k$},\\
		%		I_k,&\text{if $Q_l\subset I_k$ for some $j$}.
		%	\end{cases}\]
	Obviously, $\{D_j\}_{j\ge 1}$ is a net cover of $U$. Moreover, if $D_j$ is obtained from (b), then
	\[A\cap D_j=A\cap Q_l\subset \bigcup_{k: I_k\subset Q_l} I_k=\bigcup_{k: I_k\subset D_j} I_k,\]
	which means the rightmost union forms a cover of $A\cap D_j$. It follows that
	\[\begin{split}
		\sum_{k\ge 1}f(|I_k|)&= \sum_{k:\,\exists D_j\in \text{(a)}, s.t.\,D_j=I_k}f(|I_k|)+\sum_{k:\,\exists D_j\in \text{(b)}, s.t.\,I_k\subset D_j}f(|I_k|)\\
		&\ge \sum_{D_j\in \text{(a)}}f(|D_j|)+\sum_{D_j\in \text{(b)}}\nc^f(A\cap D_j)\\
		&\ge\sum_{D_j\in \text{(a)}}\nc^f(D_j)+\sum_{D_j\in \text{(b)}}\nc^f(D_j)\ge \nc^f(U).
	\end{split}\]
	Taking all the possible net cover of $A\cap U$, we have
	\[\nc^f (A\cap  U)=\nc^f(U).\qedhere\]
\end{proof}

Following the same lines as the proof of \cite[Lemma 4]{Bu04}, one can obtain the $\nc^f$-outer measure of countable intersection of $G_\delta$-sets, which implies that $\scg^f(X)$ is closed under countable intersection. Note that we need the above lemma in this general setting.
\begin{lem}\label{l:inters}
	Let $f$ be a dimension function, and let $\{A_n\}_{n\ge 1}$ be a sequence of $G_{\delta}$-sets. If there exists a constant $c>0$ such that for any $n\ge 1$
	\[\nc^f(A_n\cap Q)>c\nc^f(Q)\quad\text{for all $Q\in \cq$},\]
	then
	\[\nc^f\bigg(\bigcap_{n= 1}^\infty A_n\cap Q\bigg)= \nc^f(Q)\quad\text{for all $Q\in \cq$}.\]
\end{lem}

\subsection{Proof of Theorem \ref{t:LIPfull}}
Before moving on to the proof, we state and prove some auxiliary results. The following measure-theoretic result is an alternative version of a technical lemma due to Beresnevich, Dickinson and Velani \cite{BDV06}. In the presence of \eqref{eq:BsubQsubB} and the doubling property of $g$, the proof follows the same line as \cite[Lemma 7]{BDV06} with some minor modifications.

\begin{lem}\label{l:density}
	Let $A$ be a Borel subset of $X$. Assume that there exists $ c>0 $ such that for any $ Q\in\cq $, we have $ \hm^g(A\cap Q)>c\hm^g(Q) $. Then,
	\[\hm^g(A)=\hm^g(X).\]
\end{lem}

The next lemma allows us to compare the Hausdorff $f$-content with the Hausdorff $g$-measure.
\begin{lem}\label{l:fe<ge}
	Let $A$ be a Borel subset of $X$. Let $f$ be a dimension function so that $f\preceq g$. Then
	\[\hc^f(A)\apprge \frac{f(|A|)}{g(|A|)}\hm^g(A).\]
\end{lem}
\begin{proof}
	Without loss of generality, assume that $\hm^g(A)\ne 0$. Define a probability measure $\mu$ supported on $A$ by
	\[\mu=\frac{\hm^g|_{A}}{\hm^g(A)}.\]
	For any $B(x,r)$ with $x\in A$ and $r<|A|$, we have
	\[\mu\big(B(x,r)\big)=\frac{\hm^g|_{A}\big(B(x,r)\big)}{\hm^g(A)}\lesssim\frac{g(r)}{\hm^g(A)}=\frac{f(r)}{\hm^g(A)}\cdot \frac{g(r)}{f(r)}\le \frac{f(r)}{\hm^g(A)}\cdot \frac{g(|A|)}{f(|A|)},\]
	where the last inequality follows from $f\preceq g$. By the mass distribution principle below, we arrive at the conclusion.
\end{proof}
\begin{prop}[Mass distribution principle {\cite[Lemma 1.2.8]{BiPe17}}]\label{p:MDP}
	Let $ A $ be a Borel subset of $ X $. If $ A $ supports a strictly positive Borel measure $ \mu $ that satisfies
	\[\mu(B)\le cf(|B|),\]
	for some constant $ 0<c<\infty $ and for every ball $B$, then $ \hc^f(A)\ge\mu(A)/c $.
\end{prop}
\begin{rem}\label{r:strongequi}
	If $0<\lim_{r\to 0}f(r)/g(r)=C<\infty$, then we would have $f(r)/g(r)$ is bounded on $(0,+\infty)$. So $\hc^f$, $\hm^f$ and $\hm^g$ are strongly equivalent in the following sense:
	\[\hc^f(A)\asymp\hm^g(A)\qaq \hm^f(A)=C\hm^g(A)\quad\text{for all Borel set $A\subset X$}.\]
\end{rem}

We can now prove Theorem \ref{t:LIPfull}.
% Recall the condition
%\begin{equation}\label{eq:condd}
%	\lim_{n\to\infty} \hc^f \bigg(\bigcup_{k=n}^\infty E_k\cap  B\bigg)>c\hc^f(B).
%\end{equation}
%By \eqref{eq:BsubQsubB} and \eqref{eq:eqiv}, it is easily verified that the above holds with $B$ replaced by $Q\in\cq$ and with $c$ repalced by a smaller constant.
\begin{proof}[Proof of Theorem \ref{t:LIPfull}]
	By Remark \ref{r:equiv} and Lemma \ref{l:inters}, we have $\scg^f(X)$ is closed under countable intersection, which proves the first point of the theorem.

	Let $A$ be a $G_\delta$-subset in $\scg^f(X)$. The proof of the full Hausdorff measure statement is divided into two cases.

	\noindent{\textbf{Case 1:}} $0<\lim_{r\to 0}f(r)/g(r)=C<\infty$. In this case, by Remark \ref{r:strongequi}, $\hc^f$, $\hm^f$ and $\hm^g$ are strongly equivalent.
	Therefore,
	\begin{align*}
		\hm^g(A\cap Q)\asymp\hc^f(A\cap Q)\asymp \hc^f(Q)\asymp \hm^g(Q).
	\end{align*}
	On employing Lemma \ref{l:density}, we see that $A$ is of full $\hm^g$-measure. Thus, again by Remark \ref{r:strongequi},
	\[\hm^f(A)=C\hm^g(A)=C\hm^g(X)=\hm^f(X).\]

	\noindent{\textbf{Case 2:}} $\lim_{r\to 0}f(r)/g(r)=\infty$. Note that in this case, it trivially follows that $\hm^f(X)=\infty$.
	Since $\hm^f$ is a measure, by the finitely additivity, for each $n$ it follows that
	\[\begin{split}
		\hm^f(A)&=\sum_{Q\in\cq_n}\hm^f(A\cap Q)\ge\sum_{Q\in\cq_n}\hc^f(A\cap Q)\asymp\sum_{Q\in\cq_n} f(4^{-n})\asymp \frac{f(4^{-n})}{g(4^{-n})},
	\end{split}\]
	where recall that the diameter of each generalised cube in $\cq_n$ is comparable to $4^{-n}$.
	Let $n\to \infty$, the right hand side would go to infinity, which completes the proof.
\end{proof}

\section{Proof: MTP from balls to open sets}\label{s:basic}
In this section, we verify that the $\limsup$ set satisfying the condition in Theorem \ref{t:MTPBtO} belongs to $\scg^f(X)$. The proof is based on a Frostman-type lemma presented in \ref{s:weighed Hausdorff}.
\subsection{Weighted Hausdorff measure and Frostman-type lemma}\label{s:weighed Hausdorff}
A useful covering lemma which we will use throughout is the following (see \cite[Theorem 1.2]{He01} or \cite[19, Theorem 2.1]{Mat95}).
\begin{lem}[$5r$-covering lemma]\label{l:5r}
	Every family $\cf$ of balls of uniformly bounded diameter in a metric space $X$ contains a disjointed subfamily $\cg$ such that
	\[\bigcup_{B\in\cf}B\subset \bigcup_{B\in\cg}5B.\]
\end{lem}

We draw inspiration from the concepts presented in \cite{KR21}. We first define weighted Hausdorff measure and then use it to formulate a Frostman-type lemma.

For $0<\eta< \infty$ and any set $A\subset X$, set
\[\lambda^f_\eta(A)=\inf \sum_i c_i f(|A_i|),\]
where the infimum is taken over all finite or countable families $\{(A_i,c_i)\}$ such that $0<c_i<\infty$, $A_i\subset X$, $|A_i|\le\eta$ and
\[\chi_A\le \sum_i c_i\chi_{A_i}.\]
Obviously $\lambda_\eta^f(A)$ is non-increasing in $\eta$ and we define the {\em weighted Hausdorff $f$-measure} as
\[\lambda^f(A):=\lim_{\eta\to 0}\lambda_\eta^f(A).\]
Similarly, when $\eta=\infty$, $\lambda_{\infty}^f(A)$ is referred to as the {\em weighted Hausdorff $f$-content}.
The following two lemmas can be proved, respectively, in the same way as in \cite[Lemmas 8.16 and 8.17]{Mat95}, in which the case $f(r)=r^s$ was established. It is important to highlight that the compactness of $X$ is a prerequisite for both Lemmas \ref{l:relate} and \ref{l:frost}, while the doubling property of $f$ is necessary only in Lemma \ref{l:relate}. In addition, the proof of Lemma \ref{l:relate} uses the $5r$-covering lemma, which is the main reason why the constant $D^6$ appears.
\begin{lem}\label{l:relate}
	Let $f$ be a doubling dimension function. Then
	\[\hm^f(X)\le D^6\lambda^f(X)\qaq \hc^f(X)\le D^6\lambda_{\infty}^f(X),\]
	where $D$ is a constant given in \eqref{eq:doubling}.
\end{lem}

\begin{lem}\label{l:frost}
	Let $0<\eta\le \infty$. There is a Borel measure $\mu$ on $X$ such that $\mu(X)=\lambda_\eta^f(X)$ and
	\begin{equation}\label{eq:frostmu}
		\mu(A)\le f(|A|)\quad\text{for all $A\subset X$ with $|A|<\eta$}.
	\end{equation}
	In particular, if $\hm^f(X)>0$ there exist $\eta>0$ and $\mu$ satisfying \eqref{eq:frostmu} and $\mu(X)>0$.
\end{lem}
\begin{rem}\label{r:indeed doubling}
	%	Neither Lemma \ref{l:relate} nor \ref{l:frost} requires the existence of $g$-Ahlfors regular measure supported on $X$.
	Indeed, all the dimension functions $f$ under consideration are doubling, since
	\[\frac{f(2r)}{g(2r)}\le \frac{f(r)}{g(r)}\quad\Longrightarrow\quad f(2r)\le \frac{f(r)g(2r)}{g(r)}\le Df(r),\]
	by the condition $f\preceq g$ and the doubling property of $g$. Hence, Lemmas \ref{l:relate} and \ref{l:frost} are applicable to $f$ with $f\preceq g$.
\end{rem}

Lemma \ref{l:relate} together with Lemma \ref{l:frost} gives:
\begin{cor}\label{c:com}
	Let $A\subset X$ be a compact set. Let $f$ be a doubling dimension function.  Suppose that $\hc^f(A)>0$. Then, there exists a probability Borel measure $\mu$ on $A$ such that for any $r>0$
	\[\mu\big(B(x,r)\big)\le \frac{D^6f(2r)}{\hc^f(A)}.\]
\end{cor}
The above corollary is somewhat unsatisfactory because it is only valid for compact sets. To derive a Frostman-type lemma that holds for all Borel sets, we will use the following result concerning the approximation of a Borel set from within by compact sets.
\begin{thm}[{\cite[Theorem 48]{Rog70}}]
	Let $E$ be a Borel set in $X$ with $\hc^f(E)>\alpha>0$. Then $E$ has a compact subset $A$ with
	\[\hc^f(A)>\alpha.\]
	Consequently,
	\[\hc^f(E)=\sup\{\hc^f(A):A\subset E\text{ is compact}\}.\]
\end{thm}
This together with Corollary \ref{c:com} implies:
\begin{prop}\label{p:frost}
	Let $f$ be a doubling dimension function. Let $A\subset X$ be a Borel set. Then, there exists a probability Borel measure $\mu$ on $A$ such that for any $r>0$,
	\[\mu\big(B(x,r)\big)\le \frac{D^8f(r)}{\hc^f(A)}.\]
\end{prop}

With Proposition \ref{p:frost} now at our disposal, we can prove Theorem \ref{t:MTPBtO} in the next subsections: We first construct a suitable subset of $F_n=\bigcup_{k\ge n}E_k$, then define a mass distribution $\mu$ supported on it and estimate the $\mu$-measure of arbitrary balls.

\subsection{Construction of the subset of $F_n\cap B$}
The following $K_{G,B}$-lemma comes from \cite[Lemma 5]{BV06} will be required in the present setting.
\begin{lem}[{$K_{G,B}$-lemma}]\label{l:kgb}
	Let $\{B_n\}_{n\ge 1}$ be a sequence of balls in $X$ with radii tending to $0$. For any ball $B$ in $X$ suppose that $\hm^g(\limsup B_n\cap B)=\hm^g(B)$. Then for any ball $B$ and any $G>1$ there is a finite sub-collection
	\[K_{G,B}\subset \{B_n:n\ge G\}\]
	such that
	\begin{enumerate}
		\item all the balls in $K_{G,B}$ are contained in $B$;
		\item the balls are $3r$-disjoint in the sense that for different balls $B_n,B_k\in K_{G,B}$
		\[3B_n\cap 3B_k=\emptyset.\]
		\item there exists a constant $c>0$ independent of our choice of ball $B$ such that
		\[\hm^g\bigg(\bigcup_{B_n\in K_{G,B}}B_n\bigg)>c\hm^g(B).\]
	\end{enumerate}
\end{lem}

Suppose the assumption of Theorem \ref{t:MTPBtO} holds. That is
\[\hc^f(E_n)>c\hm^g(B_n)\quad\text{ for all $n\ge 1$}.\]
Write $F_n=\bigcup_{k\ge n}E_k$. Since $\limsup F_n= \limsup E_n$, in the spirit of Theorem \ref{t:LIPfull}, to conclude Theorem \ref{t:MTPBtO} it suffices to show that for any $n\ge 1$
\begin{equation}\label{eq:f(F_n)}
	\hc^f(F_n\cap B)\apprge\hc^f(B)\quad\text{for all ball $B\subset X$}.
\end{equation}
Since $f\preceq g$, there are two situations to consider:
\begin{enumerate}[(a)]
	\item $\lim_{r\to 0}f(r)/g(r)=C$, where $0<C<\infty$;
	\item $f\prec g$.
\end{enumerate}

If we are in Case (a), then by Remark \ref{r:strongequi}, $\hc^f\asymp\hm^g$. By applying the $K_{G,B}$-lemma (Lemma \ref{l:kgb}), it follows that for any $n\ge 1$,
\[\hc^f(F_n\cap B)\asymp \hm^g(F_n\cap B)\asymp\hm^g\bigg(\bigcup_{B_k\in K_{G,B}}B_k\bigg)\asymp\hm^g(B)\asymp \hc^f(B),\]
which verifes \eqref{eq:f(F_n)}.

 Thus, in the remaining of this section, we assume that $f\prec g$.

Fix a ball $B$ in $X$. For each $k$, let
\[(E_k)_{-\eta}=\Big\{x\in E_k:\inf_{y\notin E_k} \dist(x,y)>\eta\Big\}.\]
Since $(E_k)_{-\eta}$ is monotonic increasing to $E_k$ as $\eta\to 0$, by Increasing Set Lemma (see Rogers \cite[Theorem 52]{Rog70})
\[\lim_{\eta\to 0}\nc^f\big((E_k)_{-\eta}\big)=\nc^f(E_k).\]
Therefore, there exists $\eta_k$ so that
\[\nc^f(E_k)\asymp\nc^f\big((E_k)_{-\eta_k}\big)\asymp\hc^f\big((E_k)_{-\eta_k}\big)\asymp \hc^f(E_k)\asymp\hm^g(B_k).\]
 Let $L_k=(E_k)_{-\eta_k}$. Namely, we choose an open set $L_k\subset E_k$ such that the Hausdorff $f$-content of it and its $\eta_k$-neighbourhood is still comparable to $E_k$. The reason we use $L_k$ instead of $E_k$ is given in Remark \ref{r:reasonLk} below.

 The desired subset of $F_n\cap B$ will take on the following form
\[K:=\bigcup_{l=1}^{l_B}\bigcup_{L_k\in K(l)} L_k,\]
where $l_B=f(|B|)/g(|B|)$ and $K(l)$ is a collection of open sets. We will also require that $K$ satisfies the following properties.

\noindent(P1) For any $l\in\{1,\dots, l_B\}$, $K(l)$ is a finite union of open sets. Moreover, for any different elements $L_k,L_{k'}\in K(l)$, the distance between $L_k$ and $L_{k'}$ is
\[\dist(L_k,L_{k'})\ge 3\max\big(r(B_k),r(B_{k'})\big).\]

\noindent(P2) For any $l\in\{1,\dots, l_B\}$, we have
\[\sum_{L_k\in K(l)} \hm^g(B_k)\asymp \hm^g(B),\]
where the unspecified constant is absolute.

\noindent(P3) For any $l\in\{1,\dots, l_B\}$, we have
\[\sum_{L_k\in K(l)} \hm^g\big(\Delta(L_k,\eta_k)\big)\le 3^{-l}\hm^g(1/2B),\]
where recall that $\Delta(L_k,\eta_k)$ is the $\eta_k$-neighbourhood of $L_k$.

\noindent(P4) For any $l\in\{1,\dots, l_B\}$, and any $L_k\in K(l)$, $L_j\in K(l')$ with $l'>l$, we have
%\[\dist(L_k,L_j)\ge \eta_k/2\qaq f(|B_j|)\le 2^{-(l'-l)}f(\eta_k/2).\]
\[\dist(L_k,L_j)\ge \eta_k/2\qaq r(B_j)<\eta_k/2.\]
\begin{rem}\label{r:reasonLk}
	To make sure the open sets from distinct levels are well-separated (see (P4)), we require that the open sets $L_j$ from higher levels do not belong to the $(\eta_k/2)$-neighbourhood of $L_k$ that from lower level. By the definition of $\eta_k$, $\Delta(L_k,\eta_k/2)$ is still a subset of $E_k$, and so its Hausdorff content is still manageable. However, $\Delta(E_k,\eta)$ could be very large for abitrary $\eta>0$. If this happens, then (P3) may not hold. That is main reason we use $L_k$ instead of $E_k$ throughout the construction.
\end{rem}

We now prove the existence of such subset $K$.

\noindent\textbf{First level}  Let $G>0$ be an integer that will be specified later. Applying the $K_{G,B}$-lemma (Lemma \ref{l:kgb}) to $B$, we obtain a finite collection $K_{G,B}$ of balls. The $\hm^g$-measure of the union of balls in $K_{G,B}$ satisfies
\[\hm^g\bigg(\bigcup_{B_k\in K_{G,B}}B_k\bigg)\asymp \hm^g(B).\]
The first level $K(1)$ is defined by
\[K(1):=\bigcup_{B_k\in K_{G,B}}L_k.\]
Note that there is a one-to-one corresponding between $K_{G,B}$ and $K(1)$. By Lemma \ref{l:fe<ge} we have
\[\begin{split}
	\sum_{L_k\in K(1)} \hm^g\big(\Delta(L_k,\eta_k)\big)&\le \sum_{L_k\in K(1)} \hm^g(E_k)\lesssim \sum_{L_k\in K(1)}\hc^f(E_k)\frac{g(|E_k|)}{f(|E_k|)}\\
	&\lesssim \sum_{L_k\in K(1)}\hm^g(B_k)\frac{g(|E_k|)}{f(|E_k|)}\asymp\hm^g(B/2)\frac{g(|E_k|)}{f(|E_k|)}.
\end{split}\]
Since $f\prec g$ and the implied constant is absolute, we may suppose that $G$ is large enough so that the right hand side is less than
\[3^{-1}\hm^g(1/2B).\]
%Let
%\[r_1=\min\{\eta_k:L_k\in K(1)\}.\]
It is easily verified that $K(1)$ satisfies (P1)--(P4). This finishes the construction of the first level.

\noindent\textbf{Higher levels} The higher levels are defined inductively. Suppose the first $l$ levels $K(1),\dots, K(l)$ have been constructed and properties (P1)--(P4) hold with $l$ in place of $l_B$. Since we require the open sets in each level is well-separated (see (P1) and (P4)), we first verify that there is ``space" left over in $B$ for the level $K(l+1)$ after the first $l$ levels, $K(1),\dots,K(l)$, have been constructed. Let
\begin{equation}\label{eq:Al}
	A^{(l)}:=\frac{1}{2}B\setminus \bigg(\bigcup_{i=1}^{l}\bigcup_{L_k\in K(i)} \Delta(L_k,\eta_k)\bigg),
\end{equation}
where we recall that $\Delta(L_k,\eta_k)\subset E_k$.
Using property (P3), we obtain
\[\begin{split}
	\hm^g\bigg(\bigcup_{i=1}^{l}\bigcup_{L_k\in K(i)} \Delta(L_k,\eta_k)\bigg)&\le \sum_{i=1}^{l}3^{-i}\hm^g(1/2B)\le 2^{-1}\hm^g(1/2B).
\end{split}\]
So,
  \[\hm^g(A^{(l)})\ge 2^{-1}\hm^g(1/2B)\asymp\hm^g(B).\]

Next, since each $K(i)$ with $1\le i\le l$ is a finite set, let
\begin{equation}\label{eq:rl+1}
	r_{\min}=\min \{\eta_k:L_k\in K(i), 1\le i\le l\}.
\end{equation}
%By decreasing $r_{\min}$ if necessary, we assume that for any $L_k\in K(i)$ with $1\le i\le l$,
%\begin{equation}\label{eq:f(rl+1)<}
%	f(2^{-1}r_{\min})\le 2^{-(l+1-i)}f(\eta_k/2).
%\end{equation}
%This is possible since $K(i)$ is finite for each $1\le i\le l$.
Define
\begin{equation}\label{eq:All}
	\ca(l):=\{B(x,r_{\min}/2):x\in A^{(l)}\}.
\end{equation}
By 5$r$-covering lemma  there exists a disjoint subcollection $\cf(l+1)$ of $\ca(l)$ such that
\begin{equation}\label{eq:3r}
	A^{(l)}\subset \bigcup_{\tilde B\in \ca(l)}\tilde B\subset \bigcup_{\tilde B\in \cf(l+1)}5\tilde B.
\end{equation}
Note that each ball in the collection $\cf(l+1)$ is contained in $B$. Since the balls in this collection are disjoint and all have the same radius, $\cf(l+1)$ must be finite. Furthermore,
\begin{equation}\label{eq:compa1}
	\hm^g\bigg(\bigcup_{\tilde B\in\cf(l+1)}\tilde B\bigg)\asymp\hm^g\bigg(\bigcup_{\tilde B\in\cf(l+1)}5\tilde B\bigg)\ge \hm^g(A^{(l)})\asymp \hm^g(B).
\end{equation}
%Now, to construct the $(l+1)$th level $K(l+1)$, let $\tilde G\ge G$ be sufficiently large so that we can apply $K_{G,B}$ lemma to each ball $\tilde B\in\cf(l+1)$. Moreover, we assume that $\tilde G$ is sufficiently large so that for every $j\ge \tilde G$ and any $B_k\in K(l)$,
%\begin{equation}\label{eq:Bj<Bk}
%	f(|B_j|)\le f(|B_k|)/2.
%\end{equation}
% 	Such a $\tilde G$ exists since there are only finite many balls in $K_l$.
To each ball $\tilde B\in \cf(l+1)$ we apply $K_{G,B}$-lemma (Lemma \ref{l:kgb}) to obtain a collection of balls $K_{\tilde G,\tilde B}$, where $\tilde G$ will be specified later. We then define
\[K(l+1):=\bigcup_{\tilde B\in\cf(l+1)}\bigcup_{B_k\in K_{\tilde G,\tilde B}}L_k.\]
Now we verify $K(l+1)$ satisfies (P1)--(P4). Property (P1) holds because of the definition of $\cf(l+1)$ (see \eqref{eq:3r}) and the separation property of $K_{\tilde G,\tilde B}$ (see Lemma \ref{l:kgb} (2)). To see (P2) holds we first observe that the measure of the union of balls $\tilde B\in\cf(l+1)$ is comparable to $\hm^g(B)$ (see \eqref{eq:compa1}). Secondly, for any ball $\tilde B\in\cf(l+1)$, the measure of union of balls in $K_{\tilde G,\tilde B}$ is comparable to $m(\tilde B)$ (see Lemma \ref{l:kgb2} (3)). Combining these two observations, we can conclude that (P2) is satisfied. For (P4), note first that by the definitions of $A^{(l)}$ and $\ca(l)$ (see \eqref{eq:Al} and \eqref{eq:All}, respectively), the center of balls in $\cf(l+1)$ does not belong to the $\eta_k$-neighbourhood of $L_k$ with $L_k\in K(i)$ and $1\le i\le l$. Moreover, the radius of balls in $\cf(l+1)$ are less than $r_{\min}/2<\eta_k/2$. This yields (P4).

We now prove that (P3) holds. By Lemma \ref{l:fe<ge} and the fact that $\Delta(L_k,\eta_k)\subset E_k$
\[\begin{split}
	\sum_{L_k\in K(l+1)} \hm^g\big(\Delta(L_k,\eta_k)\big)&\le \sum_{L_k\in K(l+1)} \hm^g(E_k)\lesssim \sum_{L_k\in K(l+1)}\hc^f(E_k)\frac{g(|E_k|)}{f(|E_k|)}\\
	&\lesssim \sum_{L_k\in K(l+1)}\hm^g(B_k)\frac{g(|E_k|)}{f(|E_k|)}\asymp\hm^g(B/2)\frac{g(|E_k|)}{f(|E_k|)}.
\end{split}\]
%\[\sum_{L_k\in K(l+1)} \hm^g\big(\Delta(L_k,\eta_k)\big)\le \sum_{L_k\in K(l+1)}\hc^f(E_k)\frac{g(|E_k|)}{f(|E_k|)}\le \sum_{E_k\in K(l+1)}c\hm^g(B_k)\frac{g(|E_k|)}{f(|E_k|)}.\]
Similarly, since $f\prec g$ and the implied constant is absolute, we may suppose that $\tilde G$ is large enough so that the right hand side is less than
\[3^{-(l+1)}\hm^g(1/2B),\]
which verifies (P3).
\subsection{A measure $\mu$ supported on $K$}
By Remark \ref{r:indeed doubling}, the dimension function $f$ is doubling. For each $k\ge 1$, we apply Proposition \ref{p:frost} to obtain a measure $\nu_k$ supported on $L_k$ so that for any $x\in L_k$ and $r>0$
\begin{equation}\label{eq:nuk}
	\nu_k\big(B(x,r)\big)\le \frac{D^8f(r)}{\hc^s(L_k)}\le \frac{D^8f(r)}{c\hm^g(B_k)}.
\end{equation}
Let
\[\mu:=\frac{1}{l_B}\sum_{l=1}^{l_B}\sum_{L_k\in K(l)}\frac{\hm^g(B_k)\cdot \nu_k}{\sum_{L_k\in K(l)}\hm^g(B_k)}.\]
Clearly, $\mu$ is a probability measure supported on $K$. Now we estimate the $\mu$-measure of arbitrary ball. Let $x\in K$ and $0<r<|B|$. It is useful to partition levels into two cases:

\noindent\textbf{Case 1}: Levels $K(l)$ for which
\[\#\{L_k\in K(l):L_k\cap B(x,r)\ne\emptyset\}=1.\]

\noindent\textbf{Case 2}: Levels $K(l)$ for which
\[\#\{L_k\in K(l):L_k\cap B(x,r)\ne\emptyset\}\ge 2.\]

\noindent{\bf Dealing with Case 1}: Let $K(l^*)$ be the first level whose intersection with $B(x,r)$ is described by Case 1. There is an unique open set $L_k$ in $K(l^*)$ satisfying $L_k\cap B(x,r)\ne\emptyset$.
Consequently, by (P2) and \eqref{eq:nuk}
\[\begin{split}
	\mu\big(B(x,r)\cap K(l^*)\big)&\le\frac{1}{l_B}\cdot\frac{\hm^g(B_k)}{\sum_{L_k\in K(l^*)}\hm^g(B_k)}\cdot\nu_k\big(B(x,r)\big)\\
	&\lesssim \frac{g(|B|)}{f(|B|)}\cdot\frac{\hm^g(B_k)}{g(|B|)}\cdot\frac{f(r)}{\hm^g(B_k)}=\frac{f(r)}{f(|B|)}.
\end{split}\]
If there are at least two levels whose intersection with $B(x,r)$ is described by Case 1, then for $l\in \text{Case 1}$ with $l>l^*$ and $L_j\in K(l)$ with $L_j\cap B(x,r)\ne\emptyset$, it follows from (P4) that
\[r>\eta_k/2>r(B_j).\]
%Again by (P4), we have
%\begin{equation}\label{eq:f(Bj)}
%	f(|B_j|)\le 2^{-(l-l^*)}f(\eta_k/2)\le 2^{-(l-l^*)}f(r).
%\end{equation}
Therefore, for any $l\in \text{Case 1}$ with $l>l^*$,
\[\begin{split}
	\mu\big(B(x,r)\cap K(l)\big)&\le\frac{1}{l_B}\cdot\frac{\hm^g(B_j)}{\sum_{L_k\in K(l)}\hm^g(B_k)}\asymp\frac{1}{l_B}\cdot\frac{g(|B_j|)}{g(|B|)}\\
	&\le\frac{1}{l_B}\cdot\frac{f(|B_j|)}{f(|B|)}\le\frac{1}{l_B}\cdot \frac{f(r)}{f(|B|)}.
\end{split}\]
Therefore, the contribution to the $\mu$-measure of $B(x,r)$ from Case 1 is:
\begin{align}
	\sum_{l\in \text{Case 1}} \mu\big(B(x,r)\cap K(l)\big)&=\mu\big(B(x,r)\cap K(l^*)\big)+\sum_{l\in \text{Case 1}\atop l\ne l^*} \mu\big(B(x,r)\cap K(l)\big)\notag\\
	&\lesssim \frac{f(r)}{f(|B|)}+\sum_{l\in \text{Case 1}\atop l\ne l^*}\frac{1}{l_B}\cdot \frac{f(r)}{f(|B|)}\lesssim \frac{f(r)}{f(|B|)}.\label{eq:f(r)}
\end{align}
%where the last inequality follows from the fact that the choice of $l$ is less than $l_B$.

\noindent{\bf Dealing with Case 2}: Let $K(l)$ be the level whose intersection with $B(x,r)$ is described by Case 2. Thus, there exist dinsctint $L_k$ and $L_{k'}$ in $K(l)$, and corresponding balls $B_k$ and $B_{k'}$ satisfying $B_k\cap B(x,r)\ne\emptyset$ and $B_{k'}\cap B(x,r)\ne\emptyset$. By (P1) we know that the balls $3B_k$ and $3B_{k'}$ are disjoint. Therefore, we see that $|B_k|\le 2r$ and $B_k\subset 3B_k\subset B(x,5r)$. By the same reasoning we also have $B_{k'}\subset 3B_{k'}\subset B(x,5r)$. Hence, the contribution to the $\mu$-measure of $B(x,r)$ from Case 2 is:
\begin{align}
	\frac{1}{l_B}\sum_{l\in \text{Case 2}} \sum_{L_k\in K(l)\atop L_k\cap B(x,r)\ne\emptyset}\frac{\hm^g(B_k)}{\sum_{L_k\in K(l)}\hm^g(B_k)}&\le \frac{1}{l_B}\sum_{l\in \text{Case 2}} \frac{\hm^g\big(B(x,5r)\big)}{\sum_{L_k\in K(l)}\hm^g(B_k)}\notag\\
	&\lesssim \frac{1}{l_B}\sum_{l\in \text{Case 2}} \frac{g(r)}{g(|B|)}\le \frac{g(r)}{g(|B|)}\le \frac{f(r)}{f(|B|)}.\label{eq:f(R)/f(B)}
\end{align}

Combining the estimates \eqref{eq:f(r)} and \eqref{eq:f(R)/f(B)} gives $\mu\big(B(x,r)\big)\lesssim f(r)/f(|B|)$, thus proving \eqref{eq:f(F_n)} as desired. This completes the proof of Theorem \ref{t:MTPBtO}.

\section{Proof: Weaker condition for the class $\scg^f(X)$}
This section is reserved for the proof of Corollary \ref{c:weaken}. Recall that $\{B_k\}_{k\ge 1}$ is a sequence of balls with $\hm^g(\limsup B_k)=\hm^g(X)$, and that $\{E_n\}_{n\ge 1}$ is a sequence of open sets.

(1) Suppose that for any ball $B_k$, one has
\[\limsup_{n\to\infty} \hc^f (E_n\cap  B_k)\apprge\hm^g(B_k).\]
From this, it is not difficult to find a strictly increasing sequence $\{n_k\}_{k\ge 1}$ such that for any $k\ge 1$,
\[\hc^f(E_{n_k}\cap B_k)\apprge\hm^g(B_k).\]
Let $F_k=E_{n_k}\cap B_k$. By Theorem \ref{t:MTPBtO} we have
\[\limsup_{k\to\infty}F_k\in\scg^f(X).\]
The same hold for $\limsup E_n$, since
\[\limsup_{k\to\infty} F_k=\bigcap_{N=1}^\infty\bigcup_{k=N}^\infty E_{n_k}\cap B_k\subset \bigcap_{N=1}^\infty\bigcup_{n=N}^\infty E_{n}=\limsup_{n\to\infty}E_n.\]
This completes the proof of the first point of the corollary.

(2) The proof follows from item (1) directly, since
\[\scg_2^f(X)=\bigcap_{h\prec f}\scg^h(X)\qaq \scg_1^s(X)=\bigcap_{t<s}\scg^t(X).\]
% We prove the case for $\scg_2^f(X)$ only as the other case $\scg_1^s(X)$ follows similarly. Suppose that for any $h\prec f$ one has
%\begin{equation}\label{cond:lowerbound}
%	\limsup_{n\to\infty} \hc^h (E_n\cap  B_k)\apprge\hm^g(B_k),
%\end{equation}
%holds for all $B_k$.
%
%By item (1) we have
%\[\limsup_{n\to \infty}E_n\in\scg^{h}(X),\]
%which by the definition of  $\scg^{h}(X)$ gives
%\[\hc^h\Big(\limsup_{n\to \infty}E_n\cap B\Big)\apprge \hc^h(B).\]
%Replacing $\hc^h$ with $\nc^h$ and using Lemma \ref{l:subto=}, we have
%\[\nc^h\Big(\limsup_{n\to \infty}E_n\cap B\Big)=\nc^h(B).\]
%This is true for all $h\prec f$, by Remark \ref{r:equiv}, we conclude that
%\[\limsup_{n\to \infty}E_n\in\scg_1^f(X).\]

\section{Application: MTP with local scaling property is implied by MTP from balls to open sets}\label{s:implied}
This section is devoted to proving Corollary \ref{c:implied}. Firstly, we use the full measure assumption to construct a $\limsup$ set defined by balls that is also full measure. Let $\tilde\Upsilon=\{\tilde\Upsilon_n\}_{n\ge 1}$ with
\[\tilde{\Upsilon}_n:=g^{-1}\bigg(\bigg(\frac{f(\Upsilon_n)}{g(\Upsilon_n)^\kappa}\bigg)^{\frac{1}{1-\kappa}}\bigg).\]
Then the condition of Theorem \ref{t:mtpbtor} is written as
	\begin{equation}\label{eq:fullass}
		\hm^g\big(\Lambda(\tilde\Upsilon)\big)=\hm^g(X),
	\end{equation}
	where recall that
	\[\Lambda(\tilde\Upsilon):=\{x\in X:x\in \Delta(\rc_n,\tilde\Upsilon_n)\text{ for i.m.\,$n\in\N$}\}.\]
For any $n$, by $5r$-covering lemma, one can find a collection $\cf_n$ of balls with center in $\rc_n$ such that
\[\Delta(\rc_n,\tilde\Upsilon_n)\subset \bigcup_{B\in \cf_n}5B\qaq r(B)=\tilde\Upsilon_n.\]
Since the balls in this collection are disjoint and all have the same radius, $\cf_n$ must be finite.

By the full measure condition \eqref{eq:fullass}, it follows that
\[\hm^g\bigg(\limsup_{n\to\infty}\bigcup_{B\in\cf_n}5B\bigg)=\hm^g(X).\]
For any $n$ and $B\in \cf_n$, let
\[E_B=5B\cap \Delta(\rc_n,\Upsilon_n).\]
It follows that
\[\Lambda(\Upsilon)\supset\limsup_{n\to\infty}\bigcup_{B\in\cf_n}5B\cap \Delta(\rc_n,\Upsilon_n)=\limsup_{n\to\infty}\bigcup_{B\in\cf_n}E_B.\]
Thus, to prove Theorem \ref{t:mtpbtor} is implied by Theorem \ref{t:MTPBtO}, it suffices to show that for any $f\preceq g$, there exists a constant $c$ such that for any $n\in \N$ and $B\in\cf_n$,
\[\hc^f(E_B)>c\hm^g(B).\]
To this end, define a probability measure $\mu$ supported on $E_B$ by
\[\mu=\frac{\hm^g|_{E_B}}{\hm^g(E_B)}.\]
For any $x\in E_B$ and $r>0$, we consider three cases according to the range of $r$.

\noindent{\textbf{Case 1}}: $r>r(B)=\tilde\Upsilon_n$. It trivilly holds that
\[\mu\big(B(x,r)\big)\le 1=\frac{g(|B|)}{g(|B|)}\lesssim\frac{g(r)}{\hm^g(B)}\lesssim\frac{f(r)}{\hm^g(B)},\]
where the last inequality follows from $f\preceq g$.

\noindent{\textbf{Case 2}}: $\Upsilon_n<r\le \tilde\Upsilon_n$. By the $\kappa$-scaling property and the definition of $\tilde\Upsilon_n$,
\[\begin{split}
	\mu\big(B(x,r)\big)&=\frac{\hm^g|_{E_B}\big(B(x,r)\big)}{\hm^g(E_B)}=\frac{\hm^g\big(B(x,r)\cap 5B\cap \Delta(\rc_n,\Upsilon_n)\big)}{\hm^g\big(5B\cap \Delta(\rc_n,\Upsilon_n)\big)}\\
	&\lesssim \frac{g(\Upsilon_n)^{1-\kappa}\cdot g(r)^\kappa }{g(\Upsilon_n)^{1-\kappa}\cdot g(\tilde\Upsilon_n)^\kappa}=\frac{g(r)^\kappa}{g(\tilde\Upsilon_n)}\cdot g(\tilde{\Upsilon}_n)^{1-\kappa}\\
	&=\frac{g(r)^\kappa}{g(\tilde\Upsilon_n)}\cdot\frac{f(\Upsilon_n)}{g(\Upsilon_n)^\kappa}\le\frac{g(r)^\kappa}{g(\tilde\Upsilon_n)}\cdot\frac{f(r)}{g(r)^\kappa}\asymp \frac{f(r)}{\hm^g(B)},
\end{split}\]
where the second to the last inequality follows from the facts that $f/g^\kappa$ is a dimension function and $r>\Upsilon_n$.

\noindent{\textbf{Case 3}}: $r<\Upsilon_n$. By the definition of $\tilde\Upsilon_n$, we have
\[\begin{split}
	\mu\big(B(x,r)\big)&=\frac{\hm^g|_{E_B}\big(B(x,r)\big)}{\hm^g(E_B)}\lesssim \frac{g(r)}{g(\Upsilon_n)^{1-\kappa}\cdot g(\tilde\Upsilon_n)^\kappa}=\frac{g(r)}{g(\tilde\Upsilon_n)}\cdot \bigg(\frac{g(\tilde{\Upsilon}_n)}{g(\Upsilon_n)}\bigg)^{1-\kappa}\\
	&=\frac{g(r)}{g(\tilde\Upsilon_n)}\cdot\frac{f(\Upsilon_n)}{g(\Upsilon_n)}\le\frac{g(r)}{g(\tilde\Upsilon_n)}\cdot\frac{f(r)}{g(r)}\asymp \frac{f(r)}{\hm^g(B)},
\end{split}\]
where the second to the last inequality follows from $f\preceq g$ and $r<\Upsilon_n$.

Combining Cases 1--3, by the mass distribution principle, we have
\[\hc^f(E_B)\apprge\hm^g(B),\]
as desired.

\section{Application: Simpler proof of MTP from rectangles to rectangles}\label{s:rtor}
This section is devoted to proving Theorem \ref{t:meaRtoR}. Recall that $X^\times=\prod_{i=1}^{d}X_i$ with $X_i$ supporting a $\delta_i$-Ahlfors regular measure $m_i$ for $1\le i\le d$, and $m^\times=m_1\times\cdots\times m_d$. Recall also that
\[J_n=\{\alpha\in J:l_n\le\beta(\alpha)\le u_n\},\quad\ca=\{a_1,\dots,a_d,a_1+t_1,\dots,a_d+t_d\}\]
and we assume that $a_1$ is the smallest one in $\ca$ and $a_d+t_d$ is the largest one.  Let
\begin{equation}\label{eq:En}
	E_n:=\bigcup_{\alpha\in J_n}\Delta\big(\rcat,\rho(u_n)^{\bm a+\bm t}\big).
\end{equation}
Let $s=s(\bm t)$ be the dimensional number given in Theorem \ref{t:mRtoR}. Without loss of generality assume that not every $t_i$ ($1\le i\le d$) is $0$. We have $s<\sum_{i=1}^{d}\delta_k$.
In view of Corollary \ref{c:weaken} (1), the desired conclusion follows from the stronger statement below: There exists a constant $c>0$ such that
\[\limsup_{n\to\infty} \hc^s (E_n\cap  B)>cm^\times(B),\]
holds for all balls $B$ in $X^\times$.

Now we borrow an auxiliary result from \cite[Lemma 8.1]{WW21} for later use.

\begin{lem}[$K_{G,B}$-lemma, {\cite[Lemma 8.1]{WW21}}]\label{l:kgb2}
	Assume the local ubiquity condition for rectangles. Let $B$ be a ball in $X^\times$ and $G\in \N$. For infinitely many $n\in \N$ with $n\ge G$, there exists a finite sub-collection $K_{G,B}$ of the rectangles
	\[\biggl\{R_\alpha=\prod_{i=1}^dB\big(z_i,\rho(u_n)^{a_i}\big):(a_1,\dots,a_d)\in\rcat,\alpha\in J_n,R_\alpha\subset B\biggr\}\]
	such that
	\begin{enumerate}[(1)]
		\item all the rectangles in $K_{G,B}$ are contained in $B$;
		\item  the rectangles are $3r$-disjoint in the sense that for any different elements in $K_{G,B}$
		\[3\prod_{i=1}^dB\big(z_i,\rho(u_n)^{a_i}\big)\cap 3\prod_{i=1}^dB\big(z_i',\rho(u_n)^{a_i}\big)=\emptyset;\]
		\item these rectangles almost pack the ball $B$ in the sense that, for a universal constant $c'>0$ depending only on the constant in the local ubiquity property,
		\[m^\times\bigg(\bigcup_{R_\alpha\in K_{G,B}}R_\alpha\bigg)>c'm^\times(B).\]
	\end{enumerate}
\end{lem}
Here the center and sidelengths of $R_\alpha$ are not explicitly written in the notation, since the former is not important to us, while the latter is implicit in the subscript $\alpha$.

The proof of Theorem \ref{t:meaRtoR} follows the idea of \cite{WW21} but with some modifications. It is easily seen that our proof is notably simpler than Wang and Wu's \cite[\S 6.1]{WW21}.

%At first, in Section \ref{s:suit} we construct a subset of $F_n\cap B$; secondly, in Section \ref{s:holder} we define a suitable mass distribution $\mu$ supported on $F_n\cap B$ and estimate the $\mu$-measure of arbitrary ball, which implies the desired conclusion.

\subsection{Proof of Theorem \ref{t:meaRtoR}}\label{s:suit}
Fix a ball $B$ in $X^\times$. Apply the $K_{G,B}$-lemma (Lemma \ref{l:kgb2}) to $B$, for infinitely many $n$ there exists a collection of well separated rectangles with the form
\[R_{\alpha}=\prod_{i=1}^dB\big(z_i,\rho(u_{n})^{a_i}\big),\]
where $z=(z_1,\dots,z_d)\in\rcat$ and $\alpha\in J_{n}$. Moreover, the $m^\times$-measure of the union of these rectanlges satisfies
\[m^\times\bigg(\bigcup_{R_\alpha\in K_{G,B}}R_\alpha\bigg)\asymp m^\times(B).\]
We stress that all the rectangles in $K_{G,B}$ have the same sidelengths.

Fix arbitrary $n$ satisfying the above properties. Write $r_0=\rho(u_{n})$.
 Recalling \eqref{eq:En}, the desired subset of $E_n\cap B$ is defined as
 \[\bigcup_{R_{\alpha}\in K_{G,B}}R_{\alpha}\cap\Delta(\rcat,r_0^{\bm a+\bm t})\subset E_n\cap B.\]
For any $R_\alpha\in K_{G,B}$, write
\[L_\alpha=R_\alpha\cap\Delta(\rcat,r_0^{\bm a+\bm t}).\]
Define
\begin{equation}\label{eq:measuremu}
	\mu:=\sum_{R_\alpha\in K_{G,B}}\frac{m^\times(R_\alpha)}{\sum_{R_\alpha\in K_{G,B}}m^\times(R_\alpha)}\cdot \frac{m^\times|_{L_\alpha}}{m^\times(L_\alpha)}.
\end{equation}
Clearly, $\mu$ is a probability measure supported on $E_n\cap B$.
Now we estimate the $\mu$-measure of arbitrary ball. Let $B(x,r)=\prod_{i=1}^dB(x_i,r)$ be a ball in $X^\times$ with $0<r<|B|$. Suppose that $x$ belongs to some $\tilde L_{\tilde\alpha}\subset \tilde R_{\tilde \alpha}\in K_{G,B}$. The estimation is devided into three cases.

\noindent\textbf{Case 1:} $r<r_{0}^{a_d+t_d}$. In this case, by the separated property of rectangles in $K_{G,B}$ (Lemma \ref{l:kgb2} (2)), we see that $B(x,r)$ only intersects the rectangle $\tilde R_{\tilde\alpha}$ to which $x$ belongs. By $\kappa_i$-scaling property and Lemma \ref{l:kgb2} (3),
\begin{align}
	\frac{m^\times(\tilde R_{\tilde\alpha})}{\sum_{R_\alpha\in K_{G,B}}m^\times(R_\alpha)}\cdot \frac{1}{m^\times(\tilde L_{\tilde\alpha})}\asymp&\frac{\prod_{i=1}^dr_0^{\delta_ia_i}}{m^\times(B)}\cdot\frac{1}{\prod_{i=1}^dr_0^{\delta_ia_i\kappa_i}\cdot r_0^{\delta_i(a_i+t_i)(1-\kappa_i)}}\notag\\
	=&\frac{1}{m^\times(B)}\cdot\frac{1}{\prod_{i=1}^dr_0^{\delta_it_i(1-\kappa_i)}}.\label{eq:useful}
\end{align}
Then, by the definition of $\mu$ (see \eqref{eq:measuremu}),
\begin{align*}
	\mu\bigg(B(x,r)\cap\bigcup_{R_\alpha\in K_{G,B}}L_\alpha\bigg)&=\frac{m^\times(\tilde R_{\tilde\alpha})}{\sum_{R_\alpha\in K_{G,B}}m^\times(R_\alpha)}\cdot \frac{m^\times|_{L_\alpha}\big(B(x,r)\big)}{m^\times(\tilde L_{\tilde\alpha})}\\
%	&\asymp \frac{1}{|B|^s}\cdot\frac{1}{\prod_{i=1}^dr_0^{\delta_it_i(1-\kappa_i)}}\cdot m|_{L_\alpha}\big(B(x,r)\big)\\
	&\lesssim\frac{1}{m^\times(B)}\cdot\frac{1}{\prod_{i=1}^dr_0^{\delta_it_i(1-\kappa_i)}}\cdot \prod_{i=1}^dr^{\delta_i}\\
	&\le\frac{\prod_{i=1}^dr^{\delta_i-\delta_it_i(1-\kappa_i)/(a_d+t_d)}}{m^\times(B)},
\end{align*}
where the last inequality follows from $r<r_{0}^{a_d+t_d}$.
One can see that
\[\sum_{i=1}^{d}\delta_i-\sum_{i=1}^{d}\frac{\delta_it_i(1-\kappa_i)}{a_d+t_d}\]
is just the one in Theorem \ref{t:mRtoR} defined by choosing $\tau=a_d+t_d$, since
\[\ck_1(\tau)=\{i:a_i\ge a_d+t_d\}=\emptyset\qaq \ck_2(\tau)=\{i:a_i+t_i\le a_d+t_d\}=\{1,\dots,d\}.\]
Thus,
\begin{equation}\label{eq:case11}
	\mu\bigg(B(x,r)\cap\bigcup_{R_\alpha\in K_{G,B}}L_\alpha\bigg)\lesssim\frac{r^s}{m^\times(B)}.
\end{equation}

\noindent\textbf{Case 2:} $r>r_0^{a_1}$. In this case, the ball $B(x,r)$ is sufficently large so that for any rectangle $R_\alpha\in K_{G,B}$,
\[B(x,r)\cap R_\alpha\ne\emptyset\quad\Longrightarrow\quad R_\alpha\subset B(x,3r).\]
By the definition of $\mu$, one has
\begin{align}
	\mu\big(B(x,r)\big)&\le\sum_{R_\alpha\in K_{G,B}\atop R_{\alpha}\subset B(x,3r)}\frac{m^\times(R_\alpha)}{\sum_{R_\alpha\in K_{G,B}}m^\times(R_\alpha)}\lesssim\frac{m^\times\big(B(x,3r)\big)}{m^\times(B)}\notag\\
	&\asymp \frac{r^{\delta_1+\cdots+\delta_d}}{m^\times(B)}\le \frac{r^s}{m^\times(B)}.\label{eq:case22}
\end{align}

\noindent\textbf{Case 3:} Arrange the elements in $\ca$ in non-descending order. Suppose that $r_0^{\tau_{k+1}}\le r<r_0^{\tau_k}$ with $\tau_k$ and $\tau_{k+1}$ are two consecutive terms in $\ca$. By \eqref{eq:useful},
\begin{align}
	&\mu\bigg(B(x,r)\cap\bigcup_{R_\alpha\in K_{G,B}}L_\alpha\bigg)\notag\\
	&\le \sum_{R_\alpha\in K_{G,B}\atop R_{\alpha}\cap B(x,r)\ne\emptyset}\frac{m^\times(R_\alpha)}{\sum_{R_\alpha\in K_{G,B}}m^\times(R_\alpha)}\cdot \frac{m|_{L_\alpha}\big(B(x,r)\big)}{m^\times(L_\alpha)}\notag\\
	&\asymp \frac{1}{m^\times(B)}\cdot\frac{1}{\prod_{i=1}^dr_0^{\delta_it_i(1-\kappa_i)}}\sum_{R_\alpha\in K_{G,B}\atop R_{\alpha}\cap B(x,r)\ne\emptyset}m^\times\big(B(x,r)\cap L_\alpha\big).\label{eq:step1}
\end{align}
Next, we estimate the number of rectangles $R_\alpha$ in $K_{G,B}$ that intersects $B(x,r)$ and the $\mu$-measure of the corresponding intersection $B(x,r)\cap L_\alpha$. For this purpose, define the sets
\[\ck_1(\tau_{k+1})=\{i:a_i\ge \tau_{k+1}\},\ \ck_2(\tau_k)=\{i:a_i+t_i\le \tau_k\},\]
\[\ck_3=\{1,\dots,d\}\setminus\big(\ck_1(\tau_{k+1})\cup\ck_2(\tau_k)\big).\]
It should be noticed the above sets we are going to use are slightly different to those given in Theorem \ref{t:mRtoR}.
% However, as explained in \cite[(4.5)]{WW21}, this makes no difference on the Hausdorff measure $\hm^s(W(\bm t))$. We shall further discuss this problem at the end of this subsection.

Note that $m^\times$ is a product measure and the sets
\[B(x,r)\cap L_\alpha=\prod_{i=1}^{d}B(x_i,r)\cap B(z_i,r_0^{a_i}) \cap\Delta(\rci,r_0^{a_i+t_i})\]
 are also product sets. Hence, the estimation builds upon the following observations:

 \noindent{\textbf{Observation A}}: in the directions $i\in\ck_1(\tau_{k+1})$,
 \[r\ge r_0^{\tau_{k+1}}\ge r_0^{a_i}.\]
 So the total number $N_i$ of rectangles that intersecting $B(x,r)$ along the $i$th direction is majorized by
 \begin{equation}\label{eq:Ni}
 	N_i\lesssim \bigg(\frac{r}{r_0^{a_i}}\bigg)^{\delta_i}.
 \end{equation}
 Moreover, for any $R_\alpha$ that intersects $B(x,r)$, by $\kappa_i$-scaling property
 \begin{align}
 	&\prod_{i\in\ck_1(\tau_{k+1})}m_i\big(B(x_i,r)\cap B(z_i,r_0^{a_i})\cap\Delta(\rci,r_0^{a_i+t_i})\big)\notag\\
 	\le&\prod_{i\in\ck_1(\tau_{k+1})}m_i\big(B(z_i,r_0^{a_i})\cap\Delta(\rci,r_0^{a_i+t_i})\big)\notag\\
 	\asymp&\prod_{i\in\ck_1(\tau_{k+1})}r_0^{\delta_ia_i\kappa_i}\cdot r_0^{\delta_i(a_i+t_i)(1-\kappa_i)}=\prod_{i\in\ck_1(\tau_{k+1})}r_0^{\delta_ia_i+\delta_it_i(1-\kappa_i)}.\label{eq:mea1}
 \end{align}

 \noindent{\textbf{Observation B}}: in the directions $i\in\ck_2(\tau_k)$,
 \[r< r_0^{\tau_k}\le r_0^{a_i+t_i}.\]
 By the speration property of $K_{G,B}$, the ball $B(x,r)$ only intersect one rectangles in the $i$th direction. Clearly,
 \begin{equation}\label{eq:mea2}
 	\prod_{i\in\ck_2(\tau_k)}m_i\big(B(x_i,r)\cap B(z_i,r_0^{a_i})\cap\Delta(\rci,r_0^{a_i+t_i})\big)\lesssim\prod_{i\in\ck_2(\tau_k)}r^{\delta_i}.
 \end{equation}

 \noindent{\textbf{Observation C}}: in the directions $i\in\ck_3$,
 \[r_0^{a_i+t_i}\le r\le r_0^{a_i}.\]
 By the speration property of $K_{G,B}$, the ball $B(x,r)$ only intersects one rectangle in the $i$th direction. By the $\kappa_i$-scaling property,
 \begin{equation}\label{eq:mea3}
 	\prod_{i\in\ck_3}m_i\big(B(x_i,r)\cap B(z_i,r_0^{a_i})\cap\Delta(\rci,r_0^{a_i+t_i})\big)\lesssim\prod_{i\in\ck_3}r^{\delta_i\kappa_i}\cdot r_0^{\delta_i(a_i+t_i)(1-\kappa_i)}.
 \end{equation}
Thus, by \eqref{eq:Ni}--\eqref{eq:mea3}
\[\begin{split}
	&\sum_{R_\alpha\in K_{G,B}\atop R_{\alpha}\cap B(x,r)\ne\emptyset}m^\times\big(B(x,r)\cap R_\alpha\cap\Delta(\rca,r_0^{\bm a+\bm t})\big)\\
	\lesssim &\prod_{i\in\ck_1(\tau_{k+1})}N_ir_0^{\delta_ia_i+\delta_it_i(1-\kappa_i)}\cdot \prod_{i\in\ck_2(\tau_k)}r^{\delta_i}\cdot\prod_{i\in\ck_3}r^{\delta_i\kappa_i}\cdot r_0^{\delta_i(a_i+t_i)(1-\kappa_i)}\\
	\lesssim& \prod_{i\in\ck_1(\tau_{k+1})}r^{\delta_i}\cdot r_0^{\delta_it_i(1-\kappa_i)}\cdot \prod_{i\in\ck_2(\tau_k)}r^{\delta_i}\cdot\prod_{i\in\ck_3}r^{\delta_i\kappa_i}\cdot r_0^{\delta_i(a_i+t_i)(1-\kappa_i)}.
\end{split}\]
This together with \eqref{eq:step1} gives
\[\begin{split}
	&\mu\bigg(B(x,r)\cap\bigcup_{R_\alpha\in K_{G,B}}L_\alpha\bigg)\\
	\lesssim&\frac{1}{m^\times(B)}\prod_{i\in\ck_1(\tau_{k+1})}r^{\delta_i}\cdot \prod_{i\in\ck_2(\tau_k)}r^{\delta_i}\cdot r_0^{-\delta_it_i(1-\kappa_i)}\cdot\prod_{i\in\ck_3}r^{\delta_i\kappa_i}\cdot r_0^{\delta_ia_i(1-\kappa_i)}:=\frac{r^t}{m^\times(B)},
\end{split}\]
where
\begin{align}
	t=\sum_{i\in\ck_1(\tau_{k+1})}\delta_i+&\sum_{i\in\ck_2(\tau_k)}\delta_i+\sum_{i\in\ck_3}\delta_i\kappa_i\notag\\
	+&\frac{\big(\sum_{i\in\ck_3}\delta_ia_i(1-\kappa_i)-\sum_{i\in\ck_2(\tau_k)}\delta_it_i(1-\kappa_i)\big)\log r_0}{\log r}.\label{eq:funr}
\end{align}
Finally, we will show that $t\ge s$, which together with \eqref{eq:case11} and \eqref{eq:case22} completes the proof. The proof of $t\ge s$ is the same as \cite[(4.5)]{WW21}, we include it for reader's convenience.

\begin{lem}
	Let $t$ and $s$ be as above. We have $t\ge s$.
\end{lem}
\begin{proof}
	It is easily seen that the term in \eqref{eq:funr} regarded as a function of $r$ is monotonic on the interval $[r_0^{\tau_{k+1}},r_0^{\tau_k}]$. So the minimal value attains when $r=r_0^{\tau_{k+1}}$ or $r_0^{\tau_k}$. Naturally, we consider two cases.

	\noindent\textbf{Case 1:} minimum attains when $r=r_0^{\tau_k}$. Then,
	\[\begin{split}
		t\ge \sum_{i\in\ck_1(\tau_{k+1})}\delta_i+\sum_{i\in\ck_2(\tau_k)}\delta_i+\sum_{i\in\ck_3}\delta_i\kappa_i+\frac{\sum_{i\in\ck_3}\delta_ia_i(1-\kappa_i)-\sum_{i\in\ck_2(\tau_k)}\delta_it_i(1-\kappa_i)}{\tau_k}.
	\end{split}\]
	Since
	\[\ck_1(\tau_k)=\ck_1(\tau_{k+1})\cup\{i:a_i=\tau_k\}\qaq \ck_3(\tau_k)=\ck_3\setminus\{i:a_i=\tau_k\},\]
	we have
\begin{align}
	&\sum_{i\in\ck_1(\tau_{k+1})}\delta_i+\sum_{i\in\ck_3}\bigg(\delta_i\kappa_i+\frac{\delta_ia_i(1-\kappa_i)}{\tau_k}\bigg)\notag\\
	=&\sum_{i\in\ck_1(\tau_{k+1})}\delta_i+\sum_{i\in\ck_3(\tau_k)}\bigg(\delta_i\kappa_i+\frac{\delta_ia_i(1-\kappa_i)}{\tau_k}\bigg)+\sum_{i:\, a_i=\tau_k}\delta_i\notag\\
	=&\sum_{i\in\ck_1(\tau_k)}\delta_i+\sum_{i\in\ck_3(\tau_k)}\bigg(\delta_i\kappa_i+\frac{\delta_ia_i(1-\kappa_i)}{\tau_k}\bigg).\label{eq:equal}
\end{align}
	By the definition of $s$,
	\[t\ge s.\]

	\noindent\textbf{Case 2:} minimum attains when $r=r_0^{\tau_{k+1}}$. Then,
	\[\begin{split}
		t\ge \sum_{i\in\ck_1(\tau_{k+1})}\delta_i+\sum_{i\in\ck_2(\tau_k)}\delta_i+\sum_{i\in\ck_3}\delta_i\kappa_i+\frac{\sum_{i\in\ck_3}\delta_ia_i(1-\kappa_i)-\sum_{i\in\ck_2(\tau_k)}\delta_it_i(1-\kappa_i)}{\tau_{k+1}}.
	\end{split}\]
	Since
	\[\ck_2(\tau_{k+1})=\ck_2(\tau_k)\cup\{i:a_i+t_i=\tau_{k+1}\}\qaq \ck_3(\tau_{k+1})=\ck_3\setminus\{i:a_i+t_i=\tau_{k+1}\},\]
	by the same reasoning as \eqref{eq:equal}, we still have
	\[t\ge s.\qedhere\]
\end{proof}
\section{Application: Simpler proof of MTP from rectangles to rectangles---general case}
The proof of Theorem \ref{t:Hausmeartor} is almost identical to the proof of Theorem \ref{t:meaRtoR} after a minor modification. Recall in Section \ref{s:rtor} that for any $B$ in $X^\times$, we apply the $K_{G,B}$-lemma (Lemma \ref{l:kgb2}) to $B$ and obtain a collection of rectangles that contained in some $\Delta\big(\rcat,\rho(u_n)^{\bm a}\big)$ with $\alpha\in J_n$. Although it is not explicitly stated in Lemma \ref{l:kgb2}, the statements indeed hold for any sufficient large $n$ with
\[ m^\times\bigg(B\cap\bigcup_{\alpha\in J_n}\Delta \big(\rcat, \rho(u_n)^{\bm a}\big)\bigg)>cm^\times(B).\]
Later, we distribute a measure $\mu$ supported on $B\cap\bigcup_{\alpha\in J_n}\Delta (\rcat, \rho(u_n)^{\bm a+\bm t})$ and then show that
\begin{equation}\label{eq:H>>}
	\mu\big(B(x,r)\big)\lesssim r^{s(\bm t)}/m^\times(B)
\end{equation}
holds for all $x$ and $r>0$, which gives the desired Hausdorff content bound.

The proof of Theorem \ref{t:Hausmeartor} follows this strategy closely. We sketch the proof and indicate the required modifications. But the details will not be presented here as they are similar. Let $t>0$ be such that $t\le s(\bm t_n)$ for infinitely many $n$. Fix a ball $B$ in $X$. By the uniform local ubiquity,
 \[m^\times\bigg(B\cap\bigcup_{\alpha\in J_n}\Delta \big(\rcat, \rho(u_n)^{\bm a}\big)\bigg)>cm^\times(B)\quad\text{for all $n\ge n_0(B)$}.\]
Choose a large $n$ so that the above holds and $t\le s(\bm t_n)$. For any such $n$, follow the same line as \eqref{eq:H>>}, we can distribute a measure $\mu$ supported on the `shrunk' set $B\cap\bigcup_{\alpha\in J_n}\Delta (\rcat, \rho(u_n)^{\bm a+\bm t_n})$ satisfying
 \[\mu\big(B(x,r)\big)\lesssim r^{s(\bm t_n)}/m^\times(B)\le r^t/m^\times(B).\]
 In turn, this gives
  \[\hc^t\bigg(B\cap\bigcup_{\alpha\in J_n}\Delta \big(\rcat, \rho(u_n)^{\bm a+\bm t_n}\big)\bigg)\apprge m^\times(B)\quad\text{for i.m. $n$}.\]
  Clearly, the above argument holds for all $t<\limsup_{n\to\infty}s(\bm t_n)$, by Corollary \ref{c:weaken} (2) we have
  \[W\big(\{\bm t_n\}\big)\in \scg_1^s(X).\]
  If $\limsup s(\bm t_n)$ is attained along a nonincreasing subsequence, then the above argument holds for $t=\limsup s(\bm t_n)$. Therefore, we can conclude that
  \[W\big(\{\bm t_n\}\big)\in \scg^s(X).\]
\section{Proof: MTP from rectangles to small open sets}\label{s:rtoo}
In this section, we still work on the product space $(X^\times,\dist^\times,m^\times)$. Recall that in the setting of Theorem \ref{t:MTPRtO},
\begin{equation}\label{con:full measure}
	m^\times\Big(\limsup_{n\to \infty}R_n\Big)=m^\times(X^\times),
\end{equation}
where  $R_n=\prod_{i=1}^{d}B(z_{n,i},r_n^{a_i})$ and $a_1\le \cdots\le a_d$.

When the local ubiquiy condition is replaced by full measure property, one still has the following result.
\begin{lem}[$K_{G,B}$-lemma, {\cite[Lemma 5.2]{WW21}}]\label{l:kgb1}
	Assume the full measure property \eqref{con:full measure} for rectangles. Let $B$ be a ball in $X^\times$. For any $G\in \N$, there exists a finite collection $K_{G,B}$ of the rectangles
	\[\biggl\{R_n=\prod_{i=1}^dB(z_{n,i},r_n^{a_i}):n\ge G,R_n\subset B\biggr\}\]
	such that
	\begin{enumerate}[(1)]
		\item all the rectangles in $K_{G,B}$ are contained in $B$;
		\item  the rectangles are $3r$-disjoint in the sense that for any different elements in $K_{G,B}$
		\[3\prod_{i=1}^dB(z_{k,i},r_k^{a_i})\cap3 \prod_{i=1}^dB(z_{n,i},r_n^{a_i})=\emptyset;\]
		%		\[3\prod_{i=1}^dB(z_{m,i},\psi(m)^{a_i})\cap 3\prod_{i=1}^dB(z_{n,i},\psi(n)^{a_i})=\emptyset;\]
		\item these rectangles almost pack the ball $B$ in the sense that,
		\[m^\times\bigg(\bigcup_{R_n\in K_{G,B}}R_n\bigg)>c' m^\times(B),\]
		where $c'$ does not depend on $G$ and $B$.
	\end{enumerate}
\end{lem}

Now we are ready to prove Theorem \ref{t:MTPRtO}.

\begin{proof}[Proof of Theorem \ref{t:MTPRtO}]
	For any $k\ge 1$, write
	\[F_n=\bigcup_{k=n}^\infty E_k,\]
	where recall that $E_k\subset R_k$ is an open sets satisfying $|E_k|\le r_k^{a_d}$.
	The goal is to show that for any $t<s$,
	\[\limsup_{n\to\infty}\hc^t(F_n\cap B)\apprge m^\times(B) \quad\text{for all $B$},\]
	where the implied constant does not depend on $B$. Then by Corollary \ref{c:weaken} one would have
	\[\limsup_{k\to\infty}E_k=\bigcap_{n=1}^\infty F_n\in\scg_1^s(X).\]

	Let $t<s$ and $\epsilon=s-t$. Fix a ball $B$ in $X$.
%	Let $G=G(B)$ be chosen so that for any $n\ge G$,
%	\begin{equation}\label{eq:condG}
%		r_n^{a_1\epsilon}\le m^\times(B)
%	\end{equation}
%	and
%	\begin{equation}\label{eq:condG1}
%		(2r_n)^{a_1\epsilon}<(1-2^{-a_1\epsilon})^{-1}.
%	\end{equation}
%	For any $n\ge G$, a
	Apply the $K_{G,B}$-lemma (Lemma \ref{l:kgb1}) to $F_n\cap B$, we obtain a collection $K_{G,B}$  of well-separated rectangles. This is possible since by definition $F_n$ is of full $m^\times$-measure. For each $R_k\in K_{G,B}$ and $E_k\subset R_k$, by the condition of Theorem \ref{t:MTPRtO}, we have
	\[|E_k|\le r_k^{a_d}\qaq\hc^s(E_k)>cm^\times(R_k).\]
	By Proposition \ref{p:frost}, there is a measure $\nu_k$ supported on $E_k$ such that for any $x\in E_k$ and $r>0$
	\begin{equation}\label{eq:nun}
		\nu_k\big(B(x,r)\big)\le \frac{D^8r^s}{\hc^s(E_k)}\le \frac{D^8r^s}{cm^\times(R_k)}.
	\end{equation}

	We group the rectangles in $K_{G,B}$ accoding to the range of $r_k$. For $\ell\in\N$, let
	\[\cg_\ell=\{R_k\in K_{G,B}:2^{-\ell-1}\le r_k<2^{-\ell}\}.\]
	Let $\ell_0$ be the smallest integer for which $\cg_{\ell_0}$ is non-empty.
%	A consequence of \eqref{eq:condG1} is that
%	\begin{equation}\label{eq:conseG1}
%		\sum_{\ell=\ell_0}^{\infty}2^{-\ell a_1\epsilon}=2^{-\ell_0a_1\epsilon}\cdot\sum_{\ell=1}^{\infty}2^{-\ell a_1\epsilon}\le (2r_G)^{a_1\epsilon}\cdot\sum_{\ell=1}^{\infty}2^{-\ell a_1\epsilon}<1.
%	\end{equation}
	For any $n\ge G$, define a probability measure $\mu$ on $F_n\cap B$ by
	\begin{align}
		\mu&=\frac{1}{\sum_{R_k\in K_{G,B}}m^\times(R_k)}\cdot \sum_{R_k\in K_{G,B}}m^\times(R_k)\cdot\nu_k\label{eq:exp1}\\
		&=\frac{1}{\sum_{R_k\in K_{G,B}}m^\times(R_k)}\cdot \sum_{\ell\ge \ell_0}\sum_{R_k\in \cg_\ell}m^\times(R_k)\cdot\nu_k,\label{eq:exp2}
	\end{align}
	where $\nu_k$ is given in \eqref{eq:nun}.

	Now, we estimate the $\mu$-measure of arbitrary balls. Let $B(x,r)=\prod_{i=1}^{d}B(x_i,r)$ be a ball with $0<r<|B|$. Suppose that $x$ belongs to some $\tilde E_l\subset \tilde R_l\in K_{G,B}$. Let $\ell_1$ be the unique integer for which
	\[2^{-(\ell_1+1)a_1}\le r<2^{-\ell_1 a_1}.\]
	Then,
	\begin{align}
		\mu\big(B(x,r)\big)&=\mu\bigg(B(x,r)\cap \bigcup_{\ell\ge \ell_0}\bigcup_{R_k\in \cg_\ell}E_k\bigg)\notag\\
		&\le\mu\bigg(B(x,r)\cap \bigcup_{\ell>\ell_1}\bigcup_{R_k\in \cg_\ell}E_k\bigg)+\sum_{\ell=\ell_0}^{\ell_1}\mu\bigg(B(x,r)\cap \bigcup_{R_k\in \cg_\ell}E_k\bigg)\notag\\
		&:=T_{\ell_1}+\sum_{\ell=\ell_0}^{\ell_1}S_\ell.\label{eq:twosum}
	\end{align}
	It is possible that $\ell_1<\ell_0$. If this happens, then we regard the summation in \eqref{eq:twosum} as $0$.

	For the quantity $T_{\ell_1}$, note that the ball $B(x,r)$ is so large that, by Lemma \ref{l:kgb1} (2), any rectangle $R_k\in \cg_\ell$ with $\ell> \ell_1$ intersecting $B(x,r)$ is contained in $B(x,3r)$. Thus, by \eqref{eq:exp2} and Lemma \ref{l:kgb1} (3),
	\begin{align}
		T_{\ell_1}&=\frac{1}{\sum_{R_k\in K_{G,B}}m^\times(R_k)}\cdot \sum_{\ell>\ell_1}\sum_{R_k\in \cg_\ell}m^\times(R_k)\cdot\nu_k\big(B(x,r)\big)\notag\\
		&\le\frac{1}{\sum_{R_k\in K_{G,B}}m^\times(R_k)}\cdot \sum_{\ell>\ell_1}\sum_{R_k\in \cg_\ell\atop R_k\subset B(x,3r)}m^\times(R_k)\notag\\
		&\le \frac{1}{\sum_{R_k\in K_{G,B}}m^\times(R_k)}\cdot m^\times\big(B(x,3r)\big)\notag\\
		&\asymp \frac{r^{\delta_1+\cdots+\delta_d}}{m^\times(B)}\le \frac{r^t}{m^\times(B)}.\label{eq:upper Tl}
	\end{align}

	For the quantity $S_\ell$ with $\ell\le \ell_1$, assume without loss of generality that $\ell_1\ge \ell_0$. For otherwise one would have $\ell<\ell_0$ and so by the definition of $\ell_0$, $S_\ell=0$. We consider two cases according to the range of $r$.

	\noindent{\textbf{Case 1:}} $r\le 2^{-\ell a_d}$. By the separation property of rectangles in $K_{G,B}$ (Lemma \ref{l:kgb1} (2)), the distance between two distinct rectangles in $\cg_\ell$ is at least $2^{-\ell a_d}$. Hence, we have
	\[B(x,r)\cap \displaystyle\bigcup_{R_k\in \cg_\ell}E_k\subset B(x,r)\cap\tilde E_l.\]
	It follows from \eqref{eq:nun} that
	\begin{align}
		\mu\big(B(x,r)\big)&=\frac{1}{\sum_{R_k\in K_{G,B}}m^\times(R_k)}\cdot m^\times(\tilde R_l)\cdot\nu_l\big(B(x,r)\big)\notag\\
		&\lesssim \frac{1}{\sum_{R_k\in K_{G,B}}m^\times(R_k)}\cdot m^\times(\tilde R_l)\cdot \frac{r^s}{m^\times(\tilde R_l)}\notag\\
		&\lesssim \frac{r^s}{m^\times(B)}=\frac{r^t\cdot r^{\epsilon}}{m^\times(B)}\le \frac{2^{-\ell a_d\epsilon}\cdot r^t}{m^\times(B)}.\label{eq:upp sl1}
	\end{align}

	\noindent{\textbf{Case 2:}} $2^{-\ell a_{j+1}}<r< 2^{-\ell a_j}$ with $1\le j\le d-1$. In this case, we have
	\begin{align}
		S_\ell&=\frac{1}{\sum_{R_k\in K_{G,B}}m^\times(R_k)}\cdot \sum_{R_k\in \cg_\ell}m^\times(R_k)\cdot\nu_k\big(B(x,r)\big)\notag\\
		&\le\frac{1}{\sum_{R_k\in K_{G,B}}m^\times(R_k)}\cdot \sum_{R_k\in \cg_\ell\atop R_k\cap B(x,r)\ne\emptyset}m^\times(R_k)\notag\\
		&\lesssim\frac{1}{m^\times(B)}\cdot \sum_{R_k\in \cg_\ell\atop R_k\cap B(x,r)\ne\emptyset}m^\times(R_k).\label{eq:upp sl2}
	\end{align}
	Now, we estimate the number of rectangles in $\cg_\ell$ that intersecting $B(x,r)$. By the separation property of rectangles in  $K_{G,B}$ (Lemma \ref{l:kgb1} (2)), in the $i$th direction with $i\le j$ the ball $B(x_i,r)$ only intersects one rectangle in $\cg_{\ell}$. While in the $i$th direction with $i\ge j+1$, by using the volume argument the number $N_i$ of rectangles that intersecting $B(x_i,r)$ are at most
	\begin{equation}\label{eq:card nj}
		N_i\lesssim \bigg(\frac{r}{2^{-\ell a_i}}\bigg)^{\delta_i}=r^{\delta_i}\cdot2^{\ell \delta_ia_i}.
	\end{equation}
	If $s\le\delta_{j+1}+\cdots+\delta_{d}$, since each rectangle in $\cg_\ell$ is of $m^\times$-measure comparable to $\prod_{i=1}^d2^{-\ell \delta_ia_i}$,  by \eqref{eq:card nj} we obtain
	\begin{align}
		\sum_{R_k\in \cg_\ell\atop R_k\cap B(x,r)\ne\emptyset}m^\times(R_k)&\lesssim \bigg(\prod_{i=j+1}^{d}r^{\delta_i}\cdot2^{\ell \delta_ia_i}\bigg)\cdot \bigg(\prod_{i=1}^d2^{-\ell \delta_ia_i}\bigg)\notag\\
		&\le r^{\delta_{j+1}+\cdots+\delta_d}\le r^s\label{eq:upp1}.
	\end{align}
	If $s>\delta_{j+1}+\cdots+\delta_{d}$, since $m^\times(R_k)\lesssim\hc^s(E_k)$ and $|E_k|\le r_n^{a_d}\asymp2^{-\ell a_d}$, we obtain
	\[m^\times(R_k)\lesssim \hc^s(E_k)\lesssim 2^{-\ell s a_d},\]
	and so by \eqref{eq:card nj}
	\[\sum_{R_k\in \cg_\ell\atop R_k\cap B(x,r)\ne\emptyset}m^\times(R_k)\lesssim \bigg(\prod_{i=j+1}^{d}r^{\delta_i}\cdot2^{\ell \delta_ia_i}\bigg)\cdot 2^{-\ell sa_d}.\]
	Since $a_1\le\cdots\le a_d$, it holds that
	\[\begin{split}
		\ell(\delta_{j+1}a_{j+1}+\cdots+\delta_da_d)-\ell a_ds&\le \ell a_d(\delta_{j+1}+\cdots+\delta_d)-\ell a_ds\\
		&=-\ell a_d(s-\delta_{j+1}-\cdots-\delta_d).
	\end{split}\]
	Using $2^{-\ell a_d}\le2^{-\ell a_{j+1}}\le r$, we have
	\begin{align}
		\bigg(\prod_{i=j+1}^{d}r^{\delta_i}\cdot2^{\ell \delta_ia_i}\bigg)\cdot 2^{-\ell sa_d}&\le r^{\delta_{j+1}+\cdots+\delta_d}\cdot 2^{-\ell a_d(s-\delta_{j+1}-\cdots-\delta_d)}\notag\\
		&\le r^{\delta_{j+1}+\cdots+\delta_d}\cdot r^{s-\delta_{j+1}-\cdots-\delta_d}=r^s.\label{eq:upp2}
	\end{align}
	Finally, \eqref{eq:upp sl2} together with \eqref{eq:upp1} and \eqref{eq:upp2} yields
	\begin{equation}\label{eq:upper Sl}
		S_\ell \lesssim\frac{1}{m^\times(B)}\cdot \sum_{R_k\in \cg_\ell\atop R_k\cap B(x,r)\ne\emptyset}m^\times(R_k)\lesssim\frac{r^s}{m^\times(B)}\le \frac{2^{-\ell a_i\epsilon}\cdot r^t}{m^\times(B)}.
	\end{equation}

	With the upper bounds for $T_{\ell_1}$ and $S_\ell$ presented respectively in \eqref{eq:upper Tl} and \eqref{eq:upper Sl}, by \eqref{eq:twosum} we have
	\[\begin{split}
		\mu\big(B(x,r)\big)\lesssim \frac{r^t}{m^\times(B)}+\sum_{\ell=\ell_0}^{\ell_1}\frac{2^{-\ell a_1\epsilon}\cdot r^t}{m^\times(B)}\lesssim \frac{r^t}{m^\times(B)}.
	\end{split}\]
	This is true for all $x\in F_n\cap B$ and $r>0$, from the mass distribution principle it follows that for any ball $B$ in $X$,
	\[\hc^t(F_n\cap B)\apprge m^\times(B)\quad\text{whenever $n\ge G$},\]
	where the unspecified constant is independent of $B$. Therefore, the proof is finished.
\end{proof}

\section{Application: simpler proof of MTP for dynamical Diophantine approximation}\label{s:dmmtp}
This section is devoted to proving the large intersection property of $\cw(T,\psi)$ in the sense of $\scg_1^s(X)$. The strategy is to apply Corollary \ref{c:weaken} to the sets defining $\cw(T,\psi)$.
% For reader's convenience, we briefly recall the settings of Theorem \ref{t:dmmtp2}.
%
%
%\noindent \textbf{Hypothesis A: dynamical ubiquity.} Fix $y_0\in X$ and denote
%\[Y=\bigcup_{n\ge 0} T^{-n}y_0.\]
%There exist a strictly positive continuous function $\phi\colon X\to \R^+$ and constants $0<c_2<1<c_1<\infty$ such that
%\begin{enumerate}[({A}1)]
%	\item Covering property: given $y\in Y$ and $n\in\N$, the following family of balls covers $X$
%	\[\bigl\{B(z,c_1e^{-S_n\phi(z)}):z\in T^{-n}y\bigr\}.\]
%	\item Separation property: given $y\in Y$ and $n\in \N$, the following family of balls are disjoint
%	\[\bigl\{B(z,c_2e^{-S_n\phi(z)}):z\in T^{-n}y\bigr\}.\]
%\end{enumerate}
%
%\noindent \textbf{Hypothesis B: local conformality.} For any $\lambda>1$, there exists $0<b_\lambda<1$ such that, whenever $0<b\le b_\lambda$, it holds
%\[B(Tz,\lambda^{-1}b)\subset T\big(B(z,e^{-\phi(z)}b)\big)\subset B(Tz,\lambda b)\]
%
%\noindent \textbf{Hypothesis C: exactness.} Given any ball $B$, there exists $N\in\N$ such that $T^nB=X$ whenever $n\ge N$.

Recall that $\psi\colon X\to\R^+$ is a strictly positive continuous function, and the set we are interested in are
\[\cw(T,\psi):=\big\{x\in X:\dist(x,z)<e^{-S_n(\phi+\psi)(z)}\text{ for some $z\in T^{-n}y_0$, i.m.\,$n\in\N$}\big\}.\]
Also, there is a  $\delta$-Ahlfors regular measure $m$ supported on $X$.
\subsection{Auxiliary lemmas}

We begin with some necessary parameters.

Denote by $\|\cdot\|_{\max}$ and $\|\cdot\|_{\min}$, respectively, the maximal and minimal absolute value of a continuous function over $X$. Let
\begin{align}
	\eta_1&=\frac{\log\lambda}{\|\phi\|_{\min}+\log\lambda}, &&\eta_2=\frac{2\eta_1\|\phi+\psi\|_{\max}}{\|\phi+\psi\|_{\min}},\notag\\
	\eta_3&=\frac{\log\lambda}{\|\phi\|_{\min}-\log\lambda},&&\eta_4=\eta_1(1+3\|\psi\|_{\max}+4\|\phi\|_{\max}),\notag\\
	\eta_5&=\frac{\eta_3\|\phi\|_{\max}}{\|\phi\|_{\min}}+\frac{\eta_4}{\|\phi\|_{\min}},&&\eta_6=\frac{2\eta_4}{\|\phi\|_{\min}},\notag\\
	\eta_7&=\frac{\eta_6+(\eta_3)}{1-\eta_6}\cdot\frac{\|\phi\|_{\max}}{\|\phi\|_{\min}}+\eta_5,&&\eta_8=\frac{2\eta_4+\eta_3\|\phi\|_{\max}+\eta_6\|\phi\|_{\max}}{\|\phi\|_{\min}}.\label{eq:parameter}
\end{align}
We remark that these parameters are the same as those listed in \cite[(10)]{WZ21}. Some of them will not be used here, but are still listed for comparison.
It is worth noting that all these parameters tend to 0 as $\lambda\to 1^+$. So they are small when $\lambda$ is sufficiently close to $1$.

Next, we present some notation and useful lemmas that already established in \cite{WZ21}.
\begin{lem}[{\cite[Lemma 2.1]{WZ21}}]\label{p:a_lam}
	Assume the hypothesis (C). For any $\lambda>1$, there exists $a_\lambda\in\N$ such that
	\[T^{a_\lambda}B(x,\lambda^{-1}b_\lambda)=X.\]
\end{lem}

\begin{defn}[{\cite[Definition 2.11]{WZ21}}]\label{d:t(y)}
	Given $y\in Y$ and $n\in\N$, let $t_n(y)$ be the unique solution to equation
	\begin{equation}\label{eq:t(y)}
		\sum_{z\in T^{-n}y}e^{-t\cdot S_n(\phi+\psi)(z)}=1.
	\end{equation}
\end{defn}
The following result shows that $t_n(y)$ varies uniformly over $y\in Y$ when $n$ is large.

\begin{prop}[{\cite[Proposition 2.12]{WZ21}}]\label{p:h1}
	Assume the hypotheses (A1) and (B). Let $t<t_0$ and $\lambda>1$, and recall $\eta_2$. Then there exists an integer $N_1=N_1(t,\lambda)$ such that
	\[t_n(y)\ge t(1-\eta_2)\quad\text{whenever $n\ge N_1$ and $y\in Y$}.\]
\end{prop}
\begin{lem}[{\cite[Lemma 2.14]{WZ21}}]\label{l:ln}
	Assume the hypotheses (B) and (C), and let $e^{\|\phi\|_{\min}}>\lambda>1$, $z\in X$ and $n\in\N$. Then,
	\[T^{n+[\eta_3n]+1+a_\lambda}B(z,e^{-S_n\phi(z)}b_\lambda)=X.\]
\end{lem}
\begin{lem}\label{l:variant}
	Assume the hypothesis (B) and let $z'\in B(z,6c_1e^{-S_n\phi(z)})$. Then there exists an integer $N_2=N_2(\lambda)$ such that, whenever $n\ge N_2$,
	\begin{align*}
	|S_n\phi(z)-S_n\phi(z')|&\le \sum_{i=0}^{n-1}|\phi(T^iz)-\phi(T^iz')|<n\eta_4,\\
	|S_n\psi(z)-S_n\psi(z')|&\le \sum_{i=0}^{n-1}|\psi(T^iz)-\psi(T^iz')|<n\eta_4.
	\end{align*}
\end{lem}

\subsection{Construction of suitable subset} Note that $\cw(T,\psi)$ is written as
\[\bigcap_{k=1}^\infty\bigcup_{n=k}^\infty\bigcup_{z\in T^{-n}y_0}B(z,e^{-S_n(\phi+\psi)(z)}).\]
Let
\[E_n:=\bigcup_{z\in T^{-n}y_0}B(z,e^{-S_n(\phi+\psi)(z)}).\]
Fix $\epsilon>0$ and $t<s$. By Lemma \ref{l:ln}, we choose $1<\lambda=\lambda(t,\epsilon)\ll e^{\kappa_*}$ so that
\begin{equation}\label{cond:t}
	\max\big(t(1+\eta_7),t(1+\eta_5),t/(1-\eta_8)\big)\le\delta.
\end{equation}
%Let $N=N$For any  The specify value The and $N=N(\lambda)\ge h_2(\lambda)$ so that
%\begin{equation}\label{con:t<s(1-)}
%	t<s(1-\eta_2),
%\end{equation}
%and
%\begin{equation}\label{con:ln<ne}
%	l_n:=[\eta_3n]+1+a_\lambda\le n\epsilon\quad\text{for all $n\ge N$}.
%\end{equation}
Let $N\ge N_2(\lambda)$ so that Lemma \ref{l:variant} is applicable for all $n\ge N$. The goal of this section is to verify that
\begin{equation}\label{eq:goingtopr}
	\limsup_{n\to\infty}\hc^t(E_n\cap B)\apprge m(B)
\end{equation}
holds for any ball $B$ with $r(B)\le e^{-N\|\phi\|}$.
This would implies that the above estimate still holds for arbitrary ball $B$ but with a larger contant. Therefore, Corollary \ref{c:weaken} is still applicable.

%Indeed, if $B(z,r)$ is a ball in $X$ with $r>e^{-N\phi}$, by \eqref{eq:goingtopr} it is trivial
%\begin{align*}
%	\limsup_{n\to\infty}\hc^t\big(E_n\cap B(z,r)\big)&\ge\limsup_{n\to\infty}\hc^t\big(E_n\cap B(z,e^{-N\|\phi\|})\big)\\
%	&>cm\big(B(z,e^{-N\|\phi\|})\big)>ccm\big(B(z,e^{-N\|\phi\|})\big)\cdot cm\big(B(z,e^{-N\|\phi\|})\big)|B(z,r)|^{t+\epsilon}.
%\end{align*}
%The upshot is that $N$ depends on $\lambda=\lambda(t,\epsilon)$ and hence on $t$ and $\epsilon$. Therefore the constant $c(2e^{-N\|\phi\|})^{t+\epsilon}$ depends on $t$ and $\epsilon$ only, and we arrive at the desired conclusion.

Now, let us prove \eqref{eq:goingtopr}.
%The proof is classic. We first construct a subset in this subsection, which sits inside the set $E_n\cap B$, and $n$ is large enough. Secondly, we distribute a measure supported on this subset in next subsection and obtain the H\"older exponent.
Fix a ball $B_0=B(z_0,r_0)$ with $r_0\le e^{-N\|\phi\|}$. Let $k$ be the unique integer such that
\[e^{-S_k\phi(z_0)}\le r_0<e^{-S_{k-1}\phi(z_0)}.\]
Since $r_0\le e^{-N\|\phi\|}$, we have
\[k\ge N=N(\lambda).\]
Apply Lemma \ref{l:ln} and write $l_k=[\eta_3k]+1+a_\lambda$, we see that
\begin{equation}\label{eq:whole}
	X=T^{k+l_k}B(z_0,e^{-S_k\phi(z_0)}b_\lambda)\subset T^{k+l_k}B(z_0,r_0).
\end{equation}
By \eqref{eq:whole}, for any $n\ge 1$ and $z^*\in T^{-n}y_0$, there exists $z\in B(z_0,e^{-S_k\phi(z_0)})$ such that
\[T^{k+l_k}z=z^*\text{ or equaivalently }z\in T^{-(k+l_k)}z^*\subset T^{-(k+l_k+n)}y_0.\]
There maybe multiple choice, but we only choose one of them. Following the teminalogy in \cite{WZ21}, we call $z$ the point in $T^{-k-l_k-n}y_0$ corresponding to $z^*\in T^{-n}y_0$. Define
\begin{equation}\label{eq:collection}
	\cf_n:=\Big\{B(z,e^{-S_{k+l_k+n}(\phi+\psi)(z)}):z^*\in T^{-n}y_0\Big\}\subset E_{k+l_k+n}\cap B(z_0,e^{-S_k\phi(z_0)}).
\end{equation}
For any $B_1=B(z,e^{-S_{k+l_k+n}(\phi+\psi)(z)})\in\cf_n$, by hypothesis (A2), one has that the distance between $B_1$ and any other balls from $\cf_n$ is at least
\begin{equation}\label{eq:dist2}
	\frac{c_2}{2}e^{-S_{k+l_k+n}\phi(z)}.
\end{equation}

\subsection{Suitable measure on subset}\label{s:suit2}
Recalling Proposition \ref{p:h1}, let $n_0\ge N_1(t,\lambda)$ be an integer such that for any $n\ge n_0$,
\begin{equation}\label{eq:condik}
	e^{\delta\cdot S_{k+l_k+n}(\phi+\psi)(z)}\cdot e^{-t\cdot S_n(\phi+\psi)(z^*)}\le 1,
\end{equation}
where $z^*\in T^{-n}y_0$ and $z\in T^{-k-l_k-n}y_0$ the point corresponding to $z^*$. This is possible, since $t<s\le \delta$ and $T^{k+l_k}z=z^*$.

Fix $n\ge n_0$. Define a probability measure $\mu$ supported on $E_{k+l_k+n}\cap B(z_0,e^{-S_k\phi(z_0)})$ by
\[\mu:=\sum_{z^*\in T^{-n}y_0}\frac{e^{-tS_n(\phi+\psi)(z^*)}}{\sum_{z^*\in T^{-n}y_0}e^{-tS_n(\phi+\psi)(z^*)}}\cdot \frac{m|_{B(z,e^{-S_{k+l_k+n}(\phi+\psi)(z)})}}{m\big(B(z,e^{-S_{k+l_k+n}(\phi+\psi)(z)})\big)},\]
where $z\in T^{-k-l_k-n}$ is the point corresponding to $z^*\in T^{-n}y_0$.

Now, we estimate the $\mu$-measure of arbitrary balls. Since the proof is quite similar to \cite{WZ21}, some complicated estimates will not be detailed here, only the key ideas are pointed out. We refer the interested reader to that paper for further details.

For any ball $B(x,r)$ with $x\in B(z_1,e^{-S_{k+l_k+n}(\phi+\psi)(z_1)})\in\cf_n$ and $r>0$. We consider four cases.

\noindent \textbf{Case 1}: $r\ge e^{-S_k\phi(z_0)}$. Then,
\begin{equation}\label{eq:case1}
	\mu\big(B(x,r)\big)\le 1\le\frac{r^t}{e^{-t\cdot S_k\phi(z_0)}}\asymp\frac{r^t}{r_0^t},
\end{equation}
where we use $r_0\asymp e^{-S_k\phi(z_0)}$ in the last inequality.

\noindent \textbf{Case 2}: $e^{-S_{k+l_k}\phi(z_1)}\le r< e^{-S_k\phi(z_0)}$. By Lemma \ref{l:variant} and recall that $l_k=[\eta_3k]+1+a_\lambda$ is negligible compared with $k$, one has
\begin{align}
	S_{k+l_k}\phi(z_1)\le S_k\phi(z_1)+l_k\|\phi\|_{\max}&\le S_k\phi(z_0)+n\eta_4+l_k\|\phi\|_{\max}\notag\\
	&\le (1+\eta_5)S_k\phi(z_0)+(1+a_\lambda)\|\phi\|_{\max}.\label{eq:var1}
\end{align}
Thus,
\begin{align}
	\mu\big(B(x,r)\big)&\le 1= \bigg(\frac{e^{-S_k\phi(z_0)}}{e^{-S_k\phi(z_0)}}\bigg)^{t(1+\eta_5)}\lesssim \bigg(\frac{e^{-\frac{1}{1+\eta_5}\cdot S_{k+l_k}\phi(z_1)}}{e^{-S_k\phi(z_0)}}\bigg)^{t(1+\eta_5)}\notag\\&
	\lesssim\frac{e^{-tS_{k+l_k}\phi(z_1)}}{r_0^{t(1+\eta_5)}}\le \frac{r^t}{r_0^{t(1+\eta_5)}},\label{eq:case2}
\end{align}
where the implied constant depends on $\lambda$ only.

\noindent \textbf{Case 3}: $e^{-S_{k+l_k+i+1}\phi(z_1)}\le r<e^{-S_{k+l_k+i}\phi(z_1)}$ for some $0\le i\le n-1$. We divide into two subcases.

\noindent \textbf{Subcase (3a)}: The case of $i\le \frac{\eta_6}{1-\eta_6}(k+l_k)$. In this case, note that $\eta_6$ is very small, and so $i$ is negligible compared with $k+l_k$, so by the same reason as \eqref{eq:var1} we have
\[\begin{split}
	S_{k+l_k+i+1}\phi(z_1)\le (1+\eta_7)S_k\phi(z_0)+\bigg(\frac{1+a_{\lambda}}{1-\eta_6}+1\bigg)\|\phi\|_{\max}.
\end{split}\]
%\[\begin{split}
%	S_{n+l_n+i+1}\phi(z_1)&\le S_n\phi(z_0)+n\eta_4+(l_n+i+1)\|\phi\|_{\max}\\
%	&\le S_n\phi(z_0)+n\eta_4+\bigg(\frac{l_n+\eta_6n}{1-\eta_6}+1\bigg)\|\phi\|_{\max}\\
%	&\le (1+\eta_7)S_n\phi(z_0)+\bigg(\frac{1+a_{\lambda}}{1-\eta_6}+1\bigg)\|\phi\|_{\max}.
% \end{split}\]
By a computation mimicking \eqref{eq:case2}, we have
\begin{equation}\label{eq:case3a}
	\mu\big(B(x,r)\big)\le 1= \bigg(\frac{e^{-S_k\phi(z_0)}}{e^{-S_k\phi(z_0)}}\bigg)^{t(1+\eta_7)}\lesssim\frac{r^t}{r_0^{t(1+\eta_7)}}.
\end{equation}

 \noindent \textbf{Subcase (3b)}: The case of $i\ge \frac{\eta_6}{1-\eta_6}(k+l_k)$. At this point, $i$ may be large, so other strategies need to be adopted.

 Denote by $\cg$ the set of all centers of balls in $\cf_n$ (see \eqref{eq:collection}) which intersect $B(x,r)$, and then by $\cg'$ the set of all points $z^*\in T^{-n}y_0$ such that there exists a point $z\in\cg$ corresponding to $z^*$. Obviously, $\cg$ is nonempty since $z_1\in\cg$.

 Let $\ell=i-\eta_6(k+l_k+i)$.  By \cite[Lemma 5.3]{WZ21}, for distinct $z,z'\in\cg$, we have
 \[T^{k+l_k+\ell}z\ne T^{k+l_k+\ell}z'.\]
 As a result, for each $\tilde z\in T^{-(n-\ell)}y_0$, there is at most one ball intersecting $B(x,r)$, with center $z\in \cg$ and $T^{k+l_k+\ell}z=\tilde z$, and correspondingly, there is at most one point $z^*\in\cg'$ with $z^*\in T^{-\ell}\tilde z$. Denote by $\ch\subset T^{-(n-\ell)}y_0$ the set of all $\tilde z\in T^{-(k-\ell)}y_0$ for which such $z\in\cg$ and $z^*\in\cg'$ exist. Clearly, there is a one-to-one corresponding between $\ch,\cg$ and $\cg'$. For each $\tilde z\in\ch$, we fix the corresponding points $z\in\cg$ and $z^*\in\cg'$, and in this case,
 \[z^*=T^{k+l_k}z\in T^{-\ell}\tilde z.\]
 Then,
 \begin{align}
 	\mu\big(B(x,r)\big)&\le \sum_{z\in \cg}\mu\big(B(z,e^{-S_{k+l_k+n}(\phi+\psi)(z)})\big)\notag\\
 	&\le\frac{\sum_{\tilde z\in \ch}e^{-tS_n(\phi+\psi)(z^*)}}{\sum_{\underline{z}^*\in T^{-n}y_0}e^{-tS_n(\phi+\psi)(\underline{z}^*)}}.\label{eq:nude}
 \end{align}
Now we estimate the numerator and denominator of \eqref{eq:nude} respectively.

For the denominator, let $N_1=N_1(t,\lambda)$ be given by Proposition \ref{p:h1}. If $\ell\ge N_1$ then $t_\ell(\tilde z)>t$ for any $\tilde z\in\ch$. So by \eqref{eq:t(y)}
\[\sum_{\underline{z}^*\in T^{-\ell}\tilde z}e^{-tS_\ell(\phi+\psi)(\underline{z}^*)}\ge 1.\]
If $\ell\le N_1$, it holds that
\[\sum_{\underline{z}^*\in T^{-\ell}\tilde z}e^{-tS_\ell(\phi+\psi)(\underline{z}^*)}\ge e^{-t\cdot N_1\|\phi+\psi\|_{\max}}.\]
Thus
\begin{align}
	\sum_{\underline{z}^*\in T^{-n}y_0}e^{-tS_n(\phi+\psi)(\underline{z}^*)}&= \sum_{\tilde z\in T^{-(n-\ell)}y_0}\sum_{\underline{z}^*\in T^{-\ell}\tilde z}e^{-tS_n(\phi+\psi)(\underline{z}^*)}\notag\\
	&=\sum_{\tilde z\in T^{-(n-\ell)}y_0}e^{-tS_{n-\ell}(\phi+\psi)(\tilde z)}\sum_{\underline{z}^*\in T^{-\ell}\tilde z}e^{-tS_{\ell}(\phi+\psi)(\underline{z}^*)}\notag\\
	&\ge e^{-t\cdot N_1\|\phi+\psi\|_{\max}}\cdot \sum_{\tilde z\in T^{-(n-\ell)}y_0}e^{-tS_{n-\ell}(\phi+\psi)(\tilde z)},\label{eq:nu}
\end{align}
where we use $ T^\ell \underline{z}^*=\tilde z$ in the second equality.

Now let us estimate the numerator of \eqref{eq:nude}. By the positivity of $\psi$ and the fact that $\ch\subset T^{-(n-\ell)y_0}$,
\begin{align}
	\sum_{\tilde z\in \ch}e^{-tS_n(\phi+\psi)(z^*)}&=\sum_{\tilde z\in \ch}e^{-tS_{\ell}(\phi+\psi)(z^*)}\cdot e^{-tS_{n-\ell}(\phi+\psi)(T^\ell z^*)}\notag\\
	&=\sum_{\tilde z\in \ch}e^{-tS_{\ell}(\phi+\psi)(z^*)}\cdot e^{-tS_{n-\ell}(\phi+\psi)(\tilde z)}\notag\\
	&\le \max_{\tilde z\in\ch}e^{-tS_\ell\phi(z^*)}\cdot \sum_{\tilde z\in T^{-(n-\ell)}y_0}e^{-tS_{n-\ell}(\phi+\psi)(\tilde z)}\notag\\
	&\le \max_{z\in\cg}e^{-tS_\ell\phi(T^{k+l_k}z)}\cdot \sum_{\tilde z\in T^{-(n-\ell)}y_0}e^{-tS_{n-\ell}(\phi+\psi)(\tilde z)}.\label{eq:de}
\end{align}
Combining \eqref{eq:nu} and \eqref{eq:de}, we have
\begin{align}
	\mu\big(B(x,r)\big)&\lesssim\max_{z\in\cg}e^{-tS_\ell\phi(T^{k+l_k}z)}=\bigg(\frac{e^{-S_k\phi(z_0)}}{e^{-S_k\phi(z_0)}}\bigg)^{t}\cdot \max_{z\in\cg}e^{-tS_\ell\phi(T^{k+l_k}z)}\notag\\
	&\asymp\frac{1}{r_0^t}\cdot \Big(e^{-S_k\phi(z_0)}\cdot \max_{z\in\cg}e^{-S_\ell\phi(T^{k+l_k}z)}\Big)^t,\label{eq:task}
\end{align}
where the implied constant depends on $t$ and $\lambda$.

So the next task is to compare
\[e^{-S_k\phi(z_0)}\cdot \max_{z\in\cg}e^{-S_\ell\phi(T^{k+l_k}z)}\text{ with }e^{-S_{k+l_k+i}\phi(z_1)},\]
where the latter is comparable to $r$.

Fix a point $z\in\cg$. By \eqref{eq:dist2}, we have $z\in B(x,3r)$ and $z_1\in B(x,3r)$. This gives $z\in B(z_1,6r)\subset B(z_1,6e^{-S_{k+l_k+i}\phi(z_1)})$ by the range of $r$ in this case. Then by Lemma \ref{l:variant},
\[\begin{split}
	|S_\ell\phi(T^{k+l_k}z)-S_\ell\phi(T^{k+l_k}z_1)|&\le \sum_{j=0}^{\ell-1}|\phi(T^{j+k+l_k}z)-\phi(T^{j+k+l_k}z_1)|\\
	&\le \sum_{j=0}^{k+l_k+i-1}|\phi(T^{j}z)-\phi(T^{j}z_1)|\\
	&\le \eta_4(k+l_k+i).
\end{split}\]

Recall $\ell=i-\eta_6(k+l_k+i)\ge 0$. Since $l_k$ and $ \eta_6(k+l_k+i)$ are negligible in comparison to $n$ and $\ell$, respectively, and since $z_1\in B(z_0, e^{-S_k\phi(z_0)})$ and $z\in B(z_1,6e^{-S_{k+l_k+i}\phi(z_1)})$, one can see that
\[\begin{split}
	S_{k+l_k+i}\phi(z_1)\le S_k\phi(z_0)+S_\ell\phi(T^{k+l_k}z)+\eta_8S_{k+l_k+i}\phi(z_1)+(1+a_\lambda)\|\phi\|_{\max}.
\end{split}\]
%\[\begin{split}
%	S_{n+l_n+i}\phi(z_1)&\le S_{n}\phi(z)+S_\ell\phi(T^{n+l_n}z)+(l_n+\eta_6(n+l_n+i))\|\phi\|_{\max}\\
%	&\le S_{n}g(z_0)+\eta_4n+S_\ell\phi(T^{n+l_n}z')+\eta_4(n+l_n+i)+(l_k+\eta_6(n+l_n+i))\|\phi\|_{\max}\\
%	&\le S_{n}g(z_0)+S_\ell\phi(T^{n+l_n}z')+\eta_8S_{n+l_n+i}\phi(z)+(1+a_\lambda)\|\phi\|_{\max},
%\end{split}\]
Rearrange the inequality, we have
\[(1-\eta_8)S_{k+l_k+i}\phi(z_1)\le S_k\phi(z_0)+S_\ell\phi(T^{k+l_k}z)+(1+a_\lambda)\|\phi\|_{\max},\]
which implies that
\[\begin{split}
	e^{-S_k\phi(z_0)}\cdot \max_{z\in\cg}e^{-S_\ell\phi(T^{k+l_k}z)}&\lesssim \big(e^{-S_{k+l_k+i}\phi(z)}\big)^{1-\eta_8}\lesssim r^{1-\eta_8}.
\end{split}\]
This together with \eqref{eq:task} gives
\begin{equation}\label{eq:case3b}
	\mu\big(B(x,r)\big)\lesssim \frac{r^{t(1-\eta_8)}}{r_0^t}\le \frac{r^t}{r_0^{t/(1-\eta_8)}}.
\end{equation}

\noindent \textbf{Case 4}: $r<e^{-S_{k+l_k+n}}\phi(z_1)$. In this case, by \eqref{eq:t(y)} and \eqref{eq:condik} we have
\begin{align}
	\mu\big(B(x,r)\big)&=\frac{e^{-S_n(\phi+\psi)(z^*)}}{\sum_{z^*\in T^{-n}y_0}e^{-t\cdot S_n(\phi+\psi)(z^*)}}\cdot \frac{m|_{B(z,e^{-S_{k+l_k+n}(\phi+\psi)(z)})}\big(B(x,r)\big)}{m\big(B(z,e^{-S_{k+l_k+n}(\phi+\psi)(z)})\big)}\notag\\
	&\lesssim \frac{e^{-t\cdot S_n(\phi+\psi)(z^*)}}{\sum_{z^*\in T^{-n}y_0}e^{-S_n(\phi+\psi)(z^*)}}\cdot \frac{r^\delta}{e^{-\delta\cdot S_{k+l_k+n}(\phi+\psi)(z)}}\notag\\
	&\le r^\delta\le \frac{r^t}{r_0^t}.\label{eq:case4}
\end{align}

Combining the estimates presented in \eqref{eq:case1}, \eqref{eq:case2}, \eqref{eq:case3a}, \eqref{eq:case3b} and \eqref{eq:case4}, we have
\[\mu\big(B(x,r)\big)\lesssim r^t\cdot \Big(\min\big(r_0^t, r_0^{t(1+\eta_7)},r_0^{t(1+\eta_5)},r_0^{t/(1-\eta_8)}\big)\Big)^{-1}.\]
By our choice of $\lambda$ (see \eqref{cond:t}), the above inequality becomes
\[\mu\big(B(x,r)\big)\lesssim r^t\cdot r_0^{-\delta}\lesssim r^t\cdot |B(z_0,r_0)|^{-\delta}.\]
Finally, by the mass distribution principle, one has
\[\hc^t(E_{k+l_k+n}\cap B)\apprge m(B).\]
Note that this holds for all sufficiently large $n$, and the proof is complete.
\subsection*{Acknowledgments}
The author would like to thank Professor Lingmin Liao for many useful suggestions.

\end{document}